\documentclass[10pt]{amsart}

\usepackage{amssymb,amsmath,amsthm,amscd}

\advance\oddsidemargin by -1cm
\advance\evensidemargin by -1cm
\newtheorem{theorem}{Theorem}

\newtheorem{proposition}{Proposition}

\newtheorem{lemma}{Lemma}

\theoremstyle{definition}
\newtheorem{remark}{Remark}

\newcommand{\bdm}{\begin{displaymath}}
\newcommand{\edm}{\end{displaymath}}
\newcommand{\bq}{\begin{equation}}
\newcommand{\eq}{\end{equation}}
\newcommand{\bqn}{\begin{equation*}}
\newcommand{\eqn}{\end{equation*}}

\newcommand{\eps}{\varepsilon}

\newcommand{\phw}{\tilde \psi^{wk}}
\newcommand{\Sing}{\mathrm{Sing}\,}
\newcommand{\Reg}{\mathrm{Reg}\,}

\newcommand{\norm}[1]{\left\| #1 \right\|}
\newcommand{\mklm}[1]{\left\{ #1 \right\}}

\renewcommand{\d}{\,d}
\newcommand{\N}{{\mathbb N}}

\newcommand{\R}{{\mathbb R}}

\newcommand{\T}{{\rm T}}

\renewcommand{\epsilon}{\varepsilon}
\renewcommand{\phi}{\varphi}
\renewcommand{\rho}{\varrho}

\newcommand{\Cinft}{{\rm C^{\infty}}}
\newcommand{\CT}{{\rm C^{\infty}_c}}

\newcommand{\g}{{\bf \mathfrak g}}

\newcommand{\id}{\mathrm{id}\,}

\renewcommand{\det}{\mathrm{det}\,}

\newcommand{\vol}{\text{vol}\,}

\newcommand{\Crit}{\mathrm{Crit}}

\DeclareMathOperator{\supp}{supp}

\DeclareMathOperator{\tr}{tr}
\DeclareMathOperator{\gd}{\partial}

\newcommand{\e}[1]{\,{\mathrm e}^{#1}\,}

\begin{document}

\author{Pablo Ramacher}
\title[Singular equivariant asymptotics and the moment map II]{Singular equivariant asymptotics and the moment map II } 
\address{Pablo Ramacher, Georg-August-Universit\"at G\"ottingen, Institut f\"ur Mathematik, Bunsenstr. 3-5, 37073 G\"ottingen, Germany}
%\subjclass{}
%\keywords{}
\email{ramacher@uni-math.gwdg.de}
\thanks{This research was financed by the grant RA 1370/2-1 of the German Research Foundation (DFG)}

\maketitle

\setcounter{tocdepth}{1}
\tableofcontents

\section{Introduction}

This is the second of a series of papers dealing with the asymptotic behavior of certain integrals occuring in the description of the spectrum of an invariant elliptic operator on a compact Riemannian manifold $M$ carrying the action of a compact, connected Lie group of  isometries $G$ \cite{bruening83,ramacher08,cassanas-ramacher}, and in the study of its equivariant cohomology via the moment map  $\mathbb{J}:T^\ast M \rightarrow \g^\ast$, where $T^\ast M$ and $\g$ denote the cotangent bundle of $M$ and the Lie algebra of $G$, respectively \cite{duistermaat-heckman, atiyah-bott84, witten92, berline-getzler-vergne}.  The mentioned integrals are essentially of the form
\bqn 
I(\mu)=\int_{T^\ast M\times \g} e^{i\mathbb{J}(\eta)(X)/\mu} a(\eta,X) \d \eta \d X,  \qquad \mu \to 0^+,
\eqn
where $a \in \CT(T^\ast M \times \g)$ is an amplitude, $d\eta$  a density on $T^\ast M$, and $dX$, up to a constant factor, the Lebesgue measure in $\g$. While asymptotics for $I(\mu)$ have been obtained for free group actions, one meets with serious difficulties when singular orbits are present. The reason is that, when trying to examine these integrals via the generalized stationary phase theorem in the case of general effective actions, the critical set of the phase function $\mathbb{J}(\eta)(X)$ is no longer a smooth manifold, so that, a priori, the principle of the stationary phase can not be applied in this case. Nevertheless, in what follows, we shall show how to circumvent this obstacle by partially resolving the singularities of the critical set of $\mathbb{J}(\eta)(X)$, and in this way obtain asymptotics for $I(\mu)$   with  remainder estimates  in the case of singular group actions. Similar asymptotics were already obtained in \cite{ramacher09} for orthogonal actions in Euclidean space, and the present paper globalizes those results, while applications will be treated in a forthcoming paper.

\section{Compact group actions and the moment map}

Let $M$ be a closed, connected Riemannian manifold, and $G$  a compact, connected Lie group with Lie algebra $\g$ acting on $M$ by isometries.   Consider the cotangent bundle  $\pi:T^\ast M\rightarrow M$, as well as the tangent bundle $\tau: T(T^\ast M)\rightarrow T^\ast M$, and define on $T^\ast M$ the Liouville form 
\bqn 
\Theta(\mathfrak{X})=\tau(\mathfrak{X})[\pi_\ast(\mathfrak{X})], \qquad \mathfrak{X} \in T(T^\ast M).
\eqn
We regard $T^\ast M$ as a  symplectic manifold with symplectic form 
\bqn 
\omega= d\Theta,
\eqn
and define for every $X \in \g$ the function
\bqn
J_X: T^\ast M \longrightarrow \R, \quad \eta \mapsto \Theta(\tilde{X})(\eta),
\eqn
where $\tilde X$ denotes the fundamental vector field on $T^\ast M$, respectively $M$,  generated by an element $X$ of $\g$. Note that  $ \Theta(\tilde {X})(\eta)=\eta(\tilde{X}_{\pi(\eta)})$. Indeed, put $\gamma(s)=\e{-s X} \cdot \eta$, $s \in ( -\epsilon, \epsilon)$ for some $\epsilon >0$, so that  $\gamma(0)=\eta$, $\dot {\gamma}(0) =\tilde X_\eta$. Since $\pi(\e{-sX} \cdot \eta)=\e{-sX} \cdot \pi(\eta)$, one computes
\bqn 
\pi_\ast(\tilde X_\eta)=\frac d {ds} \pi \circ \gamma(s)_{|s=0} = \frac d {ds} \e{-sX} \cdot \pi(\eta)_{|s=0} = \tilde X_{\pi(\eta)}.
\eqn
Therefore
$$ \Theta(\tilde {X})(\eta)=\tau(\tilde X_\eta)(\pi_\ast(\tilde X_\eta))=\eta(\tilde{X}_{\pi(\eta)}),$$
as asserted. The function $J_X$ is linear in $X$, and due to the invariance of the Liouville form
\bqn
\mathcal{L}_{\tilde X} \Theta = dJ_X+ \iota_{\tilde X} \omega =0, \qquad \forall X \in \g,
\eqn
where $\mathcal{L}_{\mathfrak X}$ denotes the Lie derivative. This means that  $G$ acts on $\T^\ast M$ in a Hamiltonian way. The corresponding symplectic moment map is  then given by 
\bqn
\mathbb{J}:T^\ast M\to \g^\ast,  \quad \mathbb{J}(\eta)(X)=J_X(\eta).
\eqn
 We are interested in the asymptotic behavior of integrals of the form 
\bq
\label{int}
I(\mu)=   \int _{T^\ast M}  \int_{\g} e^{i \psi( \eta,X)/\mu }   a(\eta,X)  \d X \d \eta , \qquad \mu \to 0^+,  
\eq 
where $a \in \CT(T^\ast M \times \g)$ is an amplitude, $d\eta$ a density on $T^\ast M$, and $dX$, up to a constant factor, the Lebesgue measure in $\g$, while
\bqn 
\psi(\eta,X) = \mathbb{J}(\eta)(X).
\eqn 
We would like to  study these integrals by means of the generalized stationary phase theorem, and for this we have to consider the critical set of the phase function $\psi(\eta,X)$. Let $\mklm{X_1, \dots, X_d}$ be a basis of $\g$, and write $X=\sum s_i X_i$. Due to the linear dependence of $J_X$ in $X$, 
\bqn 
\gd_{s_i} \psi (\eta,X) =J_{X_i}(\eta),
\eqn
and because of  the non-degeneracy of $\omega$, 
\bqn 
 J_{X, \ast}=0 \quad \Longleftrightarrow \quad dJ_X=-\iota_{\tilde X} \omega=0 \quad \Longleftrightarrow \quad \tilde X =0.
\eqn
Thus we see that 
 \begin{align*}
 \Crit(\psi)&=\mklm{ (\eta,X) \in {T^\ast M} \times \g: \psi_\ast  (\eta,X) =0}=\mklm{(\eta,X)\in \Omega \times \g: \tilde X_\eta=0 },
\end{align*}
    where
\bqn
  \Omega=\mathbb{J}^{-1}(0)
\eqn
represents the zero level of the moment map. Note that 
\bq
\label{eq:Ann}
\eta \in \Omega \quad \Longleftrightarrow \quad \eta _m \in \mathrm{Ann}(T_m (G\cdot m)) \quad \forall m \in M,
\eq
where $\mathrm{Ann} \, (V_m) \subset T_m^\ast M$ denotes the annihilator of a vector subspace $V_m \subset T_mM$. Now, the major difficulty in applying the generalized stationary phase theorem in our setting  stems from the fact that, due to the singular orbit structure of the underlying group action,  the zero level $\Omega$ of the moment map, and, consequently, the considered critical set $\Crit(\psi)$, are in general singular varieties. In fact,  if the $G$-action on $T^\ast M$ is not free, the considered moment map is no longer a submersion, so that $\Omega$ and the symplectic quotient $\Omega/G$ are no longer smooth. Nevertheless, it can be shown that these spaces have Whitney stratifications into smooth submanifolds, see Lerman-Sjamaar \cite{lerman-sjamaar}, and 
Ortega-Ratiu \cite{ortega-ratiu}, Theorems 8.3.1 and 8.3.2, which correspond to the stratifications of $T^\ast M$, and $M$ by orbit types, see Duistermaat-Kolk \cite{duistermaat-kolk}. In particular, if $(H_L)$ denotes the principal isotropy type of the $G$-action in $M$, 
$\Omega$ has a principal stratum given by
\bq
\label{eq:x}
\mathrm{Reg} \, \Omega =\mklm{ \eta \in \Omega: G_\eta \sim H_L},
\eq
where $G_\eta$ denotes the isotropy group of $\eta \in T^\ast M$. To see this, let $\eta \in \Omega$, and  $m=\pi(\eta)$ be such that $G_m \sim H_L$. In view of \eqref{eq:Ann} one computes  for $g \in G_m$, and  $\mathfrak{X} =\mathfrak{X}_T + \mathfrak{X}_N \in T_mM =T_m(G\cdot m) \oplus N_m (G \cdot m)$
 \bqn 
g \cdot \eta_m(\mathfrak{X})  =  \eta_m \big ( (L_{g^{-1}})_{\ast,m} ( \mathfrak{X}_N))= \eta_m(\mathfrak{X}),
\eqn 
since $G_m$ acts trivially on $N_m(G\cdot m)$, see Bredon \cite{bredon}, pages 308 and 181. But $G_\eta \subset G_{\pi(\eta)}$ for arbitrary $\eta$,  so that we conclude  
\bq
\label{eq:y}
\eta \in \Omega, \quad G_{\pi(\eta)} \sim H_L \quad \Rightarrow  \quad G_\eta =G_{\pi(\eta)}, 
\eq
and the assertion follows. Note  that the stratum $\mathrm{Reg} \,\Omega$ is an open and dense subset of $\Omega$,  and a smooth submanifold in $T^\ast M $ of codimension equal to the dimension $\kappa$ of a principal $G$-orbit in $M$. Since the Lie algebra of $G_\eta$ is given by $\g_\eta=\{X\in \g: \tilde X_\eta=0\}$, it is clear that the smooth part of $\Crit(\psi)$ corresponds to 
\bq
\label{eq:z}
\mathrm{Reg}\,  \Crit(\psi)=\mklm{(\eta,X)\in \mathrm{Reg}\, \Omega \times \g:  X\in \g_\eta},
\eq
and constitutes a submanifold of codimension $2\kappa$.
To obtain an asymptotic description of $I(\mu)$, we shall partially resolve the singularities of $\Crit(\psi)$, for which we will need a suitable $G$-invariant covering of $M$. In its construction, we shall follow Kawakubo \cite{kawakubo}, Theorem 4.20.  Thus, let $(H_1), \dots (H_L)$ denote the isotropy types of $M$, and  arrange them in such a way that 
\bqn
H_j \text{ is conjugate to a subgroup of }H_i  \quad \Rightarrow \quad i \leq j.
\eqn
Let $H\subset G$ be a closed subgroup, and $M(H)$ the union of all orbits of type $G/H$. Then $M$ has a stratification into orbit types according to  
\bqn
M=M(H_1) \cup \dots \cup M(H_L).
\eqn
By the principal orbit theorem, the set $M(H_L)$ is open and dense in $M$, while $M(H_1)$ is a closed, $G$-invariant submanifold. Denote by $\nu_1$ the normal $G$-vector bundle of $M(H_1)$, and by $f_1: \nu_1 \rightarrow M$ a $G$-invariant tubular neighbourhood of $M(H_1)$ in $M$. Take a $G$-invariant metric on $\nu_1$, and put
\bqn
{D}_t(\nu_1)=\mklm {v \in \nu_1: \norm{v} \leq t }, \qquad t >0.
\eqn
We then define the compact, $G$-invariant submanifold with boundary
\bqn
M_2=M - f_1(\stackrel{\circ}{D}_{1/2}(\nu_1)), 
\eqn
on which the isotropy type $(H_1)$ no longer occurs, and endow it with a $G$-invariant Riemannian metric with product form in a $G$-invariant collar neighborhood of $\gd M_2$ in $M_2$. Consider now the union $M_2(H_2)$ of orbits in $M_2$ of type $G/H_2$, a compact $G$-invariant submanifold of $M_2$ with boundary, and let $f_2:\nu_2 \rightarrow M_2$ be a $G$-invariant tubular neighbourhood  of $M_2(H_2)$ in $M_2$, which exists due to the particular form of the metric on $M_2$. Taking a $G$-invariant metric on $\nu_2$, we define
\bqn
M_3=M_2 - f_2(\stackrel{\circ}{D}_{1/2}(\nu_2)), 
\eqn
which constitutes a compact $G$-invariant submanifold with corners and isotropy types $(H_3), \dots (H_L)$. Continuing this way, one finally obtains for $M$ the decomposition 
\bqn  
M= f_1({D}_{1/2}(\nu_1)) \cup \dots  \cup f_L({D}_{1/2}(\nu_L)),
\eqn
where we identified  $f_L({D}_{1/2}(\nu_L))$ with $M_L$, which leads to the covering 
\bqn
M= f_1(\stackrel{\circ}{D}_{1}(\nu_1)) \cup \dots \cup f_L(\stackrel{\circ}{D}_{1}(\nu_L)),\qquad  f_L(\stackrel{\circ}{D}_{1}(\nu_L))=\stackrel{\circ} M_L.
\eqn

\section{The desingularization process}

Let us now start resolving the singularities of the critical set $\Crit(\psi)$.  For this, we will have to set up an iterative desingularization process along the strata of the underlying $G$-action, where each step in our iteration will consist of a decomposition, a monoidal transformation, and a reduction. For simplicity, we shall assume that at each iteration step the set of maximally singular orbits is connected. Otherwise each of the connected components, which might even have different dimensions,  has to be treated separately. 

\subsection*{First decomposition} As  in the previous section, let $f_k:\nu_k\rightarrow M_k$ be an invariant tubular neighborhood of $M_k(H_k)$ in 
\bdm
M_k=M-\bigcup_{i=1}^{k-1} f_i(\stackrel{\circ}{D}_{1/2}(\nu_i)),
\edm 
a manifold with corners on which $G$ acts with the isotropy types $(H_k), (H_{k+1}), \dots, (H_L)$, and put  $W_k=f_k(\stackrel{\circ}{D_1}(\nu_k))$. Introduce a partion of unity $\mklm{\chi_k}_{k=1,\dots,L}$ subordinate to the covering $\mklm{W_k}$, and define 
\bqn
I_k(\mu)=   \int _{T^\ast W_k}  \int_{\g} e^{i \psi(\eta,X)/\mu }   (a\chi_k)(\eta,X)  \d X \d \eta,   
\eqn
so that $I(\mu)=I_1(\mu)+\dots +I_L(\mu)$. 
As will be explained in Lemma \ref{lemma:Reg}, the critical set of $\psi$ is clean on the support of $a\chi_L$, so that we can apply  directly the stationary phase theorem to compute the integral $I_L(\mu)$. But if $k \in \mklm{1, \dots, L-1}$, the sets 
\begin{align*}
  \Omega_k&=\Omega \, \cap \, T^\ast W_k, \\
   \Crit_k(\psi) &=\mklm{(\eta,X)\in \Omega_k \times \g: \tilde X_\eta=0}
\end{align*}
are no longer smooth manifolds, so that the stationary phase theorem can not  a priori be applied  in this situation. Instead, we shall resolve the singularities of $\Crit_k(\psi)$, and after this  apply the principle of the stationary phase in a suitable resolution space. For this, introduce for each $x^{(k)}\in M_k(H_k)$ the decomposition
\bqn
\g=\g_{x^{(k)}}\oplus \g_{x^{(k)}}^\perp, 
\eqn
where $\g_{x^{(k)}}$ denotes the Lie algebra of the stabilizer $G_{x^{(k)}}$ of $x^{(k)}$, and $\g_{x^{(k)}}^\perp$ its orthogonal complement with respect to the scalar product $\tr (^tAB)$ in $\g$. Let further $A_1(x^{(k)}), \dots, A_{d^{(k)}}(x^{(k)})$ be an orthonormal basis of $\g_{x^{(k)}}^\perp$, and $B_1(x^{(k)}),\dots, B_{e^{(k)}}(x^{(k)})$ an orthonormal basis of $\g_{x^{(k)}}$. Consider  the isotropy algebra bundle over $M_k(H_k)$
\bqn
\mathfrak{iso} \,M_k(H_k) \rightarrow M_k(H_k),
\eqn
as well as the canonical projection
\bqn 
\pi_k: W_k \rightarrow M_k(H_k), \qquad m=f_k(x^{(k)},v^{(k)}) \mapsto x^{(k)}, \qquad x^{(k)} \in M_k(H_k), \, v^{(k)} \in (\nu_k)_{x^{(k)}},
\eqn
where $f_k(x^{(k)},v^{(k)})=(\exp_{x^{(k)}} \circ \gamma ^{(k)})( v^{(k)})$, and $\gamma^{(k)}$ is an equivariant diffeomorphism from $(\nu_k)_{x^{(k)}}$ onto its image, see Bredon \cite{bredon}, pages 306-307. We consider then the induced bundle
\bqn
\pi_k^\ast \frak{iso}\,  M_k(H_k)=\mklm {(f_k(x^{(k)},v^{(k)}),X)\in W_k \times \g: X \in \g_{x^{(k)}}},
\eqn
and denote by  
$$\Pi_k: W_k \times \g \rightarrow  \pi_k^\ast \frak{iso} \, M_k(H_k)$$
the canonical projection which is obtained by considering geodesic normal coordinates around $\pi_k^\ast \, \frak{iso} M_k(H_k)$, and  identifying  $W_k\times \g$ with a neighborhood of the zero section in  the normal bundle $N\, \pi_k^\ast \,\frak{iso} \, M_k(H_k)$. Note  that 
 the fiber of the normal bundle to $\pi^\ast \frak{iso}\,  M_k(H_k)$ at a point $(f_k(x^{(k)},v^{(k)}),X)$ can be identified with $\g_{x^{(k)}}^\perp$.  Integrating along the fibers of the normal bundle to  $\pi_k^\ast \, \frak{iso} M_k(H_k)$ we therefore obtain  for $I_k(\mu)$ the expression
\begin{gather*}
 I_k(\mu)=\int_{\pi_k^\ast \, \frak{iso} M_k(H_k)} \left [\int_{\Pi_k^{-1}(m,B^{(k)})\times T^\ast_m W_k
} e^{i\psi/\mu} a\chi_k \, \Phi_k \, \d (T^\ast_m W_k)(\eta) \, dA^{(k)}   \right ]  dB^{(k)} \d m \\
 =\int_{ M_k(H_k) }\left [\int_{ \g \times \pi_k^{-1}(x^{(k)})\times T^\ast_{\exp_{x^{(k)}} v^{(k)}} W_k
} e^{i\psi/\mu} a\chi_k \, \Phi_k \, \d (T^\ast_{\exp_{x^{(k)}} v^{(k)}}W_k)(\eta) \, dA^{(k)} \, dB^{(k)} \, dv^{(k)}  \right ]      dx^{(k)},
\end{gather*}
where 
\bqn
\gamma^{(k)} \big (\stackrel \circ D_1(\nu_k)_{x^{(k)}}\big )  \times  \g_{x^{(k)}}^\perp \times   \g_{x^{(k)}} \ni (v^{(k)},  A^{(k)},B^{(k)})\mapsto (\exp_{x^{(k)}} v^{(k)},A^{(k)}+B^{(k)})=(m,X)
\eqn
are coordinates on $ \pi_k^{-1}(x^{(k)})\times \g$, while $dm$, $dx^{(k)}$, $dA^{(k)}, dB^{(k)} ,  d v^{(k)}$, and $ \d (T^\ast_m W_k)(\eta) $ are suitable measures on $W_k$,   $M_k(H_k)$, $\g_{x^{(k)}}^\perp$, $\g_{x^{(k)}}$, $\gamma^{(k)}(\stackrel \circ D_1(\nu_k)_{x^{(k)}})$, and $T^\ast_m W_k$, respectively, such that 
\bqn \d X \d \eta \equiv\Phi_k \d (T^\ast_{\exp_{x^{(k)}} v^{(k)}}W_k)(\eta)  dA^{(k)} \, dB^{(k)} \, dv^{(k)} \, dx^{(k)},
\eqn
where $\Phi_k$ is a Jacobian.

\subsection*{First monoidal transformation}  Let now $k \in \mklm{1, \dots, L-1} $ be fixed. For the further analysis of the integral $I_k(\mu)$, we shall sucessively resolve the singularities of $\Crit_k(\psi)$, until we are in position to apply the principle of the stationary phase in a suitable resolution space. To begin with, we perform a monoidal transformation 
\bqn 
\zeta_k: B_{Z_k}( W_k \times \g) \longrightarrow W_k \times \g
\eqn
 in $W_k \times \g$ with center $Z_k= \frak{iso} \, M_k(H_k)$. For this, let us write  $ A^{(k)}(x^{(k)},\alpha^{(k)})=\sum  \alpha_i^{(k)} A_i^{(k)}(x^{(k)})$,  $ B^{(k)}(x^{(k)},\beta^{(k)})=\sum  \beta_i^{(k)} B_i^{(k)}(x^{(k)})$, and 
\bqn
 v^{(k)}= \sum _{i=1}^{c^{(k)}} q_i^{(k)} v_i^{(k)}(x^{(k)})\in \, \gamma^{(k)} \big ( \stackrel \circ D_1(\nu_k)_{x^{(k)}}\big ),
\eqn
where $\{v_1^{(k)}(x^{(k)}),\dots  ,v_{c^{(k)}}^{(k)}(x^{(k)})  \}$ denotes an orthonormal frame in $\nu_k$. With respect to these coordinates we have $Z_k=\mklm{\alpha^{(k)}=0, \, q^{(k)}=0}$, where $q^{(k)}=(q_1^{(k)}, \dots, q_{c^{(k)}}^{(k)})$, so that 
\begin{gather*}
B_{Z_k}( W_k \times \g)=\mklm{ (m,X,[t]) \in W_k \times \g \times \mathbb{RP}^{c^{(k)} + d^{(k)}-1}:  q^{(k)}_i t_j = q^{(k)}_j t_i, \, \alpha^{(k)}_i t_{c^{(k)}+j}=\alpha^{(k)}_j t_{c^{(k)}+i}  },\\
\zeta_k: (m,X,[t]) \longmapsto (m,X).
\end{gather*}
Let us  now cover $B_{Z_k}( W_k \times \g)$ with the charts $\mklm{(\phi_\rho, U_\rho)}$,  $U_\rho=B_{Z_k}( W_k \times \g)\cap (W_k \times \g \times V_\rho)$, where $V_\rho=\mklm{[t] \in  \mathbb{RP}^{c^{(k)} + d^{(k)}-1}: t_\rho\not=0}$. We obtain for $\zeta_k$ in each of the $q^{(k)}$-charts $\mklm{U_\rho}_{1\leq \rho\leq c^{(k)}}$ the expressions
\bqn
\, ^\rho \zeta_k=\zeta_k \circ \phi_\rho: ( x^{(k)},\tau_k,\, ^\rho \tilde v^{(k)},  A^{(k)}, B^{(k)}) \mapsto (\exp_{x^{(k)}} \tau_k  \,^\rho\tilde v^{(k)}, \tau_k A^{(k)}+ B^{(k)})\equiv(m,X),
\eqn
where $\tau_k \in (-1,1)$, 
\bqn
 \,^\rho \tilde v^{(k)}(x^{(k)},q^{(k)})= \gamma^{(k)} \Big ( \big (v_\rho^{(k)}(x^{(k)})+ \sum _{i\not=\rho}^{c^{(k)}} q_i^{(k)} v_i^{(k)}(x^{(k)})\big )  \Big / \sqrt{1 + \sum_{i\not=\rho} (q_i^{(k)} )^2} \Big ) \in \gamma^{(k)} (\,^\rho S_k^+)_{x^{(k)}},
\eqn
and 
$$\,^\rho S_k^+=\mklm{v \in \nu_k: v = \sum s_i v_i, s_\rho>0,  \norm{v}=1}.$$
  Note that for each $1 \leq \rho\leq c^{(k)}$, $$W_k \simeq f_k(\,^\rho S_k^+ \times (-1,1))$$ up to a set of measure zero. 
Now, for given $m \in M$, let $Z_m \subset T_mM$ be a neighborhood of zero such that $\exp_m: Z_m \longrightarrow M$ is a diffeomorphism onto its image. Then
\bqn
\qquad (\exp_m)_{\ast,v}: T_v Z_m \longrightarrow T_{\exp_mv}M, \quad v \in Z_m, 
\eqn
and $g \cdot \exp_m v= L_g (\exp_m v)=\exp_{L_g(m)}(L_g)_{\ast,m}( v)$. As a consequence, since $B^{(k)} \in \g_{x^{(k)}}$, we obtain
\begin{align*}
\widetilde{B^{(k)}}_{\exp_{x^{(k)}} \tau_k \, ^\rho \tilde v^{(k)}}&= \frac d{dt} \exp_{x^{(k)}} \big ( L_{\e{-t B^{(k)}}} \big )_{\ast, x^{(k)}}( \tau_k \, ^\rho \tilde v^{(k)})_{|t=0}= (\exp_{x^{(k)}})_{\ast, \tau_k \, ^\rho \tilde v ^{(k)}}\big (\lambda(B^{(k)}) ( \tau_k \, ^\rho \tilde v^{(k)})\big ) \\
&= \tau_k (\exp_{x^{(k)}})_{\ast, \tau_k \, ^\rho \tilde v ^{(k)}}\big (\lambda(B^{(k)})( \, ^\rho \tilde v^{(k)}) \big ),
\end{align*}
where we denoted by
\bqn
\lambda: \g_{x^{(k)}} \longrightarrow \mathfrak{gl}(\nu_{k,x^{(k)}}), \quad B^{(k)} \mapsto \frac d{dt} (L_{\e{-tB^{(k)}}})_{\ast, x^{(k)}|t=0} 
\eqn
the linear representation of $ \g_{x^{(k)}}$ in $\nu_{k,x^{(k)}}$, and made the canonical identification $T_v(\nu_{k,x^{(k)}}) \equiv \nu_{k,x^{(k)}}$ for any $v \in(\nu_k)_{x^{(k)}}$.  With $\pi(\eta)=m$ we therefore obtain for the phase function the factorization 
%\textbf{GRACIAS A DIOS} 
\begin{align*}
\psi(\eta ,X)&=\eta(\tilde X_{\pi(\eta)})=\eta \big (\widetilde{(\tau_k A^{(k)}+B^{(k)})}_{\exp_{x^{(k)}} \tau_k \, ^\rho \tilde v^{(k)}}\big )\\
&= \tau_k \Big [ \eta\big( \widetilde{A^{(k)}}_{\exp_{x^{(k)}} \tau_k \, ^\rho \tilde v^{(k)}}\big ) + \eta \big ( (\exp_{x^{(k)}})_{\ast, \tau_k \, ^\rho \tilde v ^{(k)}}[\lambda(B^{(k)}) ^\rho \tilde v^{(k)}]\big ) \Big]. 
\end{align*} % 21.03.2009
%%%%%%%%%%%% Factorization in coordinates
%\begin{align*}
%\psi(m,X)&=\sum_{i=1}^n p^{(k)}_i dq_{i,m}^{(k)}(\tilde X_m) = \sum_{i=1}^n p^{(k)}_i dq_{i}^{(k)}
%\big (\widetilde{(\tau_k A^{(k)}+B^{(k)})}_{\exp_{x^{(k)}} \tau_k \, ^\rho \tilde v^{(k)}}\big )\\
%&=\tau_k \left [ \sum_{i=1}^n p^{(k)}_i dq_i^{(k)}\big (\widetilde{A^{(k)}}_{\exp_{x^{(k)}} \tau_k \, ^\rho \tilde v^{(k)}}\big )+  \sum_{i=1}^n p^{(k)}_i dq_i^{(k)}\big (\widetilde{B^{(k)}}_{\exp_{x^{(k)}}  \, ^\rho \tilde v^{(k)}}\big ) \right ]
%\end{align*}
 Similar considerations hold for $\zeta_k$ in the $\alpha^{(k)}$-charts $\mklm{U_\rho}_{c^{(k)}+1 \leq \rho \leq c^{(k)}+d^{(k)}}$, so that we get
\bqn 
 \psi \circ (\id_{fiber} \otimes \zeta_k) =  \,^{(k)} \tilde \psi^{tot}=\tau_k \cdot  \,  ^{(k)} \phw, 
 \eqn
$ \,^{(k)} \tilde \psi^{tot}$ and $  \,  ^{(k)} \phw $ being the \emph{total} and \emph{weak transform} of the phase function $\psi$, respectively.\footnote{For an explanation of this notation, see section \ref{sec:6}.} Introducing a partition $\mklm{u_\rho}$ of unity subordinated to the covering $\mklm{U_\rho}$ now yields 
\bqn 
I_k(\mu)=\sum_{\rho=1} ^{c^{(k)}} \, ^\rho I_k(\mu)+\sum_{\rho=c^{(k)}+1} ^{d^{(k)}} \, ^\rho \tilde I_k(\mu),
\eqn 
where the integrals  $ ^\rho I_k(\mu)$ and $ ^\rho \tilde I_k(\mu)$ are given by the expressions
\begin{gather*}
\int_{ M_k(H_k) }\left [\int_{(\id_{fiber} \otimes\, ^\rho \zeta)^{-1}_k( \g \times \pi_k^{-1}(x^{(k)})\times T^\ast_{\exp_{x^{(k)}} v^{(k)}}W_k)}  (u_\rho \circ \phi_\rho) \, (\id_{fiber} \otimes \,^\rho\zeta_k )^\ast (  e^{i\psi /\mu }a\chi_k \, \right. \\ 
\left. \Phi_k  \d (T^\ast_{\exp_{x^{(k)}} v^{(k)}}W_k)(\eta) \, dA^{(k)} \, dB^{(k)} \,  d v^{(k)} )   \,  \right ]      dx^{(k)}.
\end{gather*}
As we shall see in section \ref{sec:8}, the weak transform $\, ^{(k)} \phw$ has no critical points in the $\alpha^{(k)}$-charts, which implies that the integrals $ ^\rho \tilde I_k(\mu)$ contribute to $I(\mu)$ only with higher order terms. In what follows, we shall therefore restrict ourselves to the situation where $a_k   \circ ( \id_{fiber} \otimes \zeta_k)$ has compact support in one of the $q^{(k)}$-charts. Thus we can assume $I_k(\mu)$ to be given by  
\begin{gather*}
\int_{ M_k(H_k) }\Big [\int_{\zeta^{-1}_k( \g \times \pi_k^{-1}(x^{(k)}))\times T^\ast_{\exp_{x^{(k)}} \tau_k \tilde v^{(k)}}W_k} e^{i\frac{\tau_k}\mu  \,  ^{(k)} \phw}(a\chi_k \circ (\id_{fiber} \otimes \zeta_k) ) \, \tilde \Phi_k \\  \d (T^\ast_{\exp_{x^{(k)}} \tau_k \tilde v^{(k)}}W_k)(\eta)  \, dA^{(k)} \, dB^{(k)} \, d\tilde v^{(k)} \d \tau_k \Big ]     \,  dx^{(k)}\\
=\int_{ M_k(H_k) \times (-1,1)  }\Big [\int_{ \gamma^{(k)}((S_k^+)_{x^{(k)}}) \times  \g_{x^{(k)}} \times \g_{x^{(k)}}^\perp \times T^\ast_{\exp_{x^{(k)}} \tau_k \tilde v^{(k)}}W_k} e^{i\frac{\tau_k}\mu \,  ^{(k)} \phw}(a\chi_k\circ (\id_{fiber} \otimes \zeta_k) ) \, \tilde \Phi_k \\   \d(T^\ast_{\exp_{x^{(k)}} \tau_k \tilde v^{(k)}}W_k)(\eta) \, dA^{(k)} \, dB^{(k)} \, d\tilde v^{(k)}   \Big ]  \d \tau_k \,     \d x^{(k)},
\end{gather*}
 where we skipped the index $\rho$, in particular identifying $\zeta_k$ with $\,^\rho \zeta_k$, and  took into account that 
\bqn 
\zeta^{-1}_k( \g \times \pi_k^{-1}(x^{(k)}))=\{x^{(k)}\} \times (-1,1) \times \gamma^{(k)}((S_k^+)_{x^{(k)}}) \times  \g_{x^{(k)}} \times \g_{x^{(k)}}^\perp.
\eqn
Here $d\tilde v^{(k)}$ is a suitable measure on the set $\gamma^{(k)}((S_k^+)_{x^{(k)}})$ such that 
$$\d X \d \eta \equiv \tilde \Phi_k \, \d(T^\ast_{\exp_{x^{(k)}} \tau_k \tilde v^{(k)}}W_k)(\eta) \, d A^{(k)} \, dB^{(k)} \,  d\tilde v^{(k)} \, d\tau_k \, dx^{(k)}.$$ 
Furthermore, a computation shows that  
\bqn 
\tilde \Phi_k  = |\tau_k|^{c^{(k)}+d^{(k)}-1} \, \Phi_k\circ \zeta_k. 
\eqn

\subsection*{First reduction} Let us now assume that there exists a $m \in W_k$ with orbit type $G/H_j$, and let  $x^{(k)} \in M_k(H_k), v^{(k)} \in (\nu_k)_{x^{(k)}}$ be such that $m=f_k(x^{(k)},v^{(k)})$. Since we can assume that $m$ lies in a slice at $x^{(k)}$ around the $G$-orbit of $x^{(k)}$, we have $G_m \subset G_{x^{(k)}}$, see Kawakubo \cite{kawakubo}, pages 184-185, and Bredon \cite{bredon}, page 86. Hence, $H_j\simeq G_m$ must be conjugate to a subgroup of $H_k\simeq G_{x^{(k)}}$. Now, $G$ acts on $M_k$ with the isotropy types $(H_k),(H_{k+1}), \dots, (H_L)$. The isotropy types occuring in $W_k$ are therefore those for which the corresponding isotropy groups  $H_k,H_{k+1}, \dots, H_L$ are conjugate to a subgroup of $H_k$, and we shall denote them by
\bqn
(H_k) = (H_{i_1}), (H_{i_2}), \dots, (H_L).
\eqn
 Now, for every $x^{(k)}\in M_k(H_k)$, $ (\nu_k)_{x^{(k)}}$ is an orthogonal $ G_{x^{(k)}}$-space; therefore $G_{x^{(k)}}$ acts on  $(S_k)_{x^{(k)}}$ with   isotropy types $(H_{i_2}), \dots, (H_L)$, cp. Donnelly \cite{donnelly78}, pp. 34. Furthermore,  by the invariant tubular neighborhood theorem, one has the isomorphism
\bqn
W_k/G \simeq (\nu_k)_{x^{(k)}}/ G_{x^{(k)}},
\eqn
so that $G$ acts on $S_k=\mklm{ v \in \nu_k: \norm {v}=1}$ with  isotropy types $(H_{i_2}), \dots, (H_L)$ as well. 
As will turn out, if $G$ acted on $S_k$ only with type $(H_L)$, the critical set of $^{(k)} \phw$ would be clean in the sense of Bott, and we could proceed to apply the stationary phase theorem to compute $I_k(\mu)$. But in general this will not be the case, and we are forced to continue with the iteration.

\subsection*{Second decomposition}

Let now $x^{(k)}\in M_k(H_k)$ be fixed. Since $\gamma^{(k)}: \nu_k \rightarrow \nu_k$ is an equivariant diffeomorphism onto its image,   $\gamma^{(k)}((S_k)_{x^{(k)}})$ is a compact $G_{x^{(k)}}$-manifold, and we consider the covering  
\bqn
\gamma^{(k)}((S_k)_{x^{(k)}}) =W_{ki_2} \cup \dots \cup W_{kL}, \qquad W_{ki_j}= f_{ki_j}(\stackrel \circ D_1(\nu_{ki_j})), \quad W_{kL}=\mathrm{Int} ( \gamma^{(k)}((S_k)_{x^{(k)}})_{L}),
\eqn
where $f_{ki_j}:\nu_{ki_j} \rightarrow  \gamma^{(k)}((S_k)_{x^{(k)}})_{i_j}$ is an invariant tubular neighborhood of $\gamma^{(k)}((S_k)_{x^{(k)}})_{i_j}(H_{i_j})$ in 
\bqn 
\gamma^{(k)}((S_k)_{x^{(k)}})_{i_j}=\gamma^{(k)}((S_k)_{x^{(k)}}) - \bigcup_{r=2}^{j-1} f_{ki_r}(\stackrel \circ D_{1/2}(\nu_{ki_r})),  \qquad j\geq 2, 
\eqn
and $f_{ki_j}(x^{(i_j)},v^{(i_j)})=(\exp_{x^{(i_j)}} \circ \gamma^{(i_j)})(v^{(i_j)})$, $x^{(i_j)} \in  \gamma^{(k)}((S_k)_{x^{(k)}})_{i_j} (H_{i_j})$, $ v ^{(i_j)} \in ( \nu_{ki_j}) _{x^{(i_j)}}$, $\gamma^{(i_j)}: \nu_{ki_j} \rightarrow \nu_{ki_j}$ being an equivariant diffeomorphism onto its image. 
Let further $\{\chi_{ki_j}\}$ denote a partition of the unity subordinated to the covering $\mklm{W_{ki_j}}$,  and define
\begin{align*}
I_{ki_j}(\mu) =&\int_{ M_k(H_k) \times (-1,1)  }\Big [\int_{ \gamma^{(k)}((S_k^+)_{x^{(k)}}) \times  \g_{x^{(k)}} \times \g_{x^{(k)}}^\perp \times T^\ast_{\exp_{x^{(k)}} \tau_k \tilde v^{(k)}}W_k} e^{i\frac{\tau_k}\mu \,  ^{(k)} \phw}(a\chi_k\circ (\id_{fiber} \otimes \zeta_k) ) \\ & \chi_{ki_j}
 \tilde \Phi_k \, \d(T^\ast_{\exp_{x^{(k)}} \tau_k \tilde v^{(k)}}W_k)(\eta) \, dA^{(k)} \, dB^{(k)} \, d\tilde v^{(k)}   \Big ]  \d \tau_k \,   \d x^{(k)},
\end{align*}
so that $I_k(\mu)= I_{ki_2}(\mu) + \dots + I_{kL}(\mu)$. It is important to note that the partition functions $\chi_{ki_j}$ depend smoothly on $x^{(k)}$ as a consequence of the tubular neighborhood theorem, by which in particular $\gamma^{(k)}(S_k) /G \simeq \gamma^{(k)}((S_k)_{x^{(k)}})/G_{x^{(k)}}$, and the smooth dependence in $x^{(k)}$ of the induced Riemannian metric on $\gamma^{(k)}((S_k)_{x^{(k)}})$, and  the metrics on the normal bundles $\nu_{ki_j}$.
Since $G_{x^{(k)}}$ acts on $W_{kL}$ only with type $(H_L)$, the iteration process for $I_{kL}(\mu)$ ends here. For the remaining integrals $I_{ki_j}(\mu)$ with $k < i_j < L$, let us denote by
\bqn
\frak{iso} \,\gamma^{(k)}((S_k)_{x^{(k)}})_{i_j}(H_{i_j}) \rightarrow \gamma^{(k)}((S_k)_{x^{(k)}})_{i_j}(H_{i_j})
\eqn
the isotropy algebra bundle over $\gamma^{(k)}((S_k)_{x^{(k)}})_{i_j}(H_{i_j})$, and by $\pi_{ki_j}: W_{ki_j} \rightarrow \gamma^{(k)}((S_k)_{x^{(k)}})_{i_j}(H_{i_j})$ the canonical projection.  For $x^{(i_j)} \in \gamma^{(k)}((S_k)_{x^{(k)}})_{i_j}(H_{i_j})$, consider the decomposition
\bqn
\g = \g_{x^{(k)}} \oplus \g_{x^{(k)}}^\perp =(\g_{x^{(i_j)}}\oplus \g_{x^{(i_j)}}^\perp) \oplus \g_{x^{(k)}}^\perp.
\eqn
Let further $A_1^{(i_j)}, \dots ,A_{d^{(i_j)}}^{(i_j)}$ be an orthonormal frame in $ \g_{x^{(i_j)}}^\perp$,  as well as $B_1^{(i_j)}, \dots ,B_{e^{(i_j)}}^{(i_j)}$ be an orthonormal frame in $ \g_{x^{(i_j)}}$, and $v_1^{(ki_j)}, \dots, v_{c^{(i_j)}}^{(ki_j)}$ an orthonormal frame in $(\nu_{ki_j}) _{x^{(i_j)}}$. Integrating along the fibers in a neighborhood of $\pi_{ki_j}^\ast \frak{iso} \, \gamma^{(k)}((S_k)_{x^{(k)}})_{i_j}(H_{i_j}) \subset W_{ki_j} \times \g_{x^{(k)}}$ then yields  for $ I_{ki_j}(\mu)$ the expression
\begin{align*}
& \int_{M_k(H_k)\times (-1,1)} \Big [ \int_{\gamma^{(k)}((S^+_k)_{x^{(k)}})_{i_j}(H_{i_j})} \Big [ \int_{\pi_{ki_j}^{-1} (x^{(i_j)})\times \g_{x^{(k)}}\times \g_{x^{(k)}}^\perp \times  T^\ast_{\exp _{x^{(k)}} \tau_k \exp_{x^{(i_j)}} v ^{( i_j)}} W_k } e^{i\frac{\tau_k}\mu \, ^{(k)} \phw } \times \\ &  (a\chi_k  \circ (\id_{fiber} \otimes \zeta_k) ) \chi_{ki_j} \,   \Phi_{ki_j} \, \d(T^\ast_{\exp _{x^{(k)}} \tau_k \exp_{x^{(i_j)}} v ^{( i_j)}} (W_k)(\eta)  \, d A^{(k)} \, d A^{(i_j)} \, dB^{(i_j)} \,  dv^{(i_j)}     \big ] dx^{(i_j)}  \Big ]   d\tau_k dx^{(k)},
\end{align*}
where $\Phi_{ki_j}$ is a Jacobian, and 
\bqn
\gamma^{(i_j)}\big (  \stackrel \circ D_1(\nu_{ki_j})_{x^{(i_j)}}\big )  \times \g_{x^{(i_j)}}^\perp \times   \g_{x^{(i_j)}} \ni (v^{(i_j)}, A^{(i_j)},B^{(i_j)})\mapsto (\exp_{x^{(i_j)}} v^{(i_j)},A^{(i_j)} + B^{(i_j)})=(\tilde v^{(k)},B^{(k)})
\eqn
are coordinates on $ \pi_{ki_j}^{-1}(x^{(i_j)})\times \g_{x^{(k)}}$, while $dx^{(i_j)}$, and $dA^{(i_j)}, dB^{(i_j)} ,  dv^{(i_j)} $ are suitable measures in the spaces   $\gamma^{(k)}((S_k)_{x^{(k)}})_{i_j}(H_{i_j})$, and  $\g_{x^{(i_j)}}^\perp$, $\g_{x^{(i_j)}}$, $\stackrel \circ D_1(\nu_{ki_j})_{x^{(i_j)}}$, respectively, such that we have the equality $ \tilde \Phi_k \d B^{(k)} \d \tilde v^{(k)}\equiv\Phi_{ki_j} \,  dA^{(i_j)} \, dB^{(i_j)} \,  dv^{(i_j)} \, dx^{(i_j)}$.

\subsection*{Second monoidal transformation}

Let us fix an $l$ such that $k< l < L$, and consider in the $q^{(k)}$-chart $(-1,1)\times \gamma^{(k)}(S_k^+) \times \g$ a monoidal transformation 
\bqn 
\zeta_{kl}: B_{Z_{kl}}((-1,1)\times \gamma^{(k)}(S_k^+) \times \g) \longrightarrow (-1,1)\times \gamma^{(k)}(S_k^+) \times \g
\eqn
with center
\bqn
Z_{kl}=  (-1,1)\times \frak{iso} \,\Gamma_{k,l}^+(H_l), \qquad \Gamma_{k,l}^+ = \bigcup _{x^{(k)} \in M_k(H_k)} \gamma^{(k)}( (S_k^+)_{x^{(k)}})_l.
\eqn
Writing  $ A^{(l)}(x^{(k)},x^{(l)},\alpha^{(l)})=\sum  \alpha_i^{(l)} A_i^{(l)}(x^{(k)}, x^{(l)})$,  $ B^{(l)}(x^{(k)},x^{(l)},\beta^{(l)})=\sum  \beta_i^{(l)} B_i^{(l)}(x^{(l)})$, and 
\bqn
 v^{(l)}(x^{(k)}, x^{(l)},q^{(l)})= \sum _{i=1}^{c^{(l)}} q_i^{(l)} v_i^{(kl)}(x^{(k)},x^{(l)}),
\eqn
one has $Z_{kl}=\mklm{\alpha^{(k)}=0, \, \alpha^{(l)}=0, \, q^{(l)}=0}$, which in particular shows that $Z_{kl}$ is a manifold. If we now cover $B_{Z_{kl}}((-1,1)\times \gamma^{(k)}(S_k^+) \times \g)$ with the standard charts, we shall see again in section \ref{sec:8} that   modulo higher order terms we can assume that $((a\chi_k \circ ( \id_{fiber} \otimes \zeta_k))\chi_{kl}) \circ \zeta_{kl}$ has compact support in one of the $q^{(l)}$-charts. Therefore it suffices to examine $\zeta_{kl}$ in one of these charts, in which it reads
\begin{gather*}
\zeta_{kl}: (x^{(k)},\tau_k, x^{(l)}, \tau_l, \tilde v^{(l)}, A^{(k)},  A^{(l)},  B^{(l)}) \mapsto \\\mapsto (x^{(k)},\tau_k, \exp_{x^{(l)}} \tau_l \tilde v^{(l)}, \tau_l A^{(k)}, \tau_l  A^{(l)}+ B^{(l)})\equiv(x^{(k)},\tau_k, \tilde v^{(k)},A^{(k)}, B^{(k)}),
\end{gather*}
where
\bqn
\tilde v^{(l)}(x^{(k)},x^{(l)},q^{(l)})= \gamma ^{(l)} \left ( \big (v_\rho^{(kl)}+ \sum _{i\not=\rho}^{c^{(l)}} q_i^{(l)} v_i^{(kl)}\big) \Big / \sqrt{1 +  \sum _{i\not=\rho} (q_i^{(l)})^2}\right )\in \gamma^{(l)}\big ((S_{kl}^+)_{x^{(l)}}\big ) 
\eqn
for some $\rho$. Note that $Z_{kl}$ has normal crossings with the exceptional divisor $E_k=\zeta_k^{-1}(Z_k) = \mklm{\tau_k=0}$, and that 
$$W_{kl} \simeq f_{kl} (S_{kl}^+ \times (-1,1))$$ up to a set of measure zero, where $S_{kl}$ denotes
 the sphere subbundle in $\nu_{kl}$,  and we set $S_{kl}^+=\mklm { v \in S_{kl}:  v=\sum v_i v_i^{(kl)}, \, v_\rho>0 }$. Consequently,  the phase function factorizes according to 
\bqn 
\psi \circ (\id_{fiber} \otimes (\zeta_k \circ \zeta_{kl}))= \,^{(kl)} \tilde \psi^{tot}=\tau_k \, \tau_l \cdot \,  ^{(kl)} \phw,
\eqn
which in the given charts reads
\begin{align*}
\psi(\eta,X)&=\tau_k \left [ \eta \Big (\widetilde{\tau_l A^{(k)}}_{\exp_{x^{(k)}} \tau_k   \exp_{x^{(l)}} \tau_l \tilde v^{(l)}}\Big )\right. \\
&\left. +  \eta \Big ( (\exp_{x^{(k)}})_{\ast, \tau_k  \exp_{x^{(l)}} \tau_l \tilde v^{(l)}}[\lambda(\tau_l A^{(l)}+B^{(l)})  \exp_{x^{(l)}}\tau_l \tilde v^{(l)} ] \Big ) \right] \\
&=\tau_k  \tau_l \left [ \eta \Big (\widetilde{ A^{(k)}}_{\exp_{x^{(k)}} \tau_k   \exp_{x^{(l)}} \tau_l \tilde v^{(l)}}\Big )+  \eta \Big ( (\exp_{x^{(k)}})_{\ast, \tau_k  \exp_{x^{(l)}} \tau_l \tilde v^{(l)}}[\lambda( A^{(l)})  \exp_{x^{(l)}}\tau_l \tilde v^{(l)} ] \Big ) \right. \\
&\left. +  \eta \Big ( (\exp_{x^{(k)}})_{\ast, \tau_k  \exp_{x^{(l)}} \tau_l \tilde v^{(l)}}\big [ (\exp_{x^{(l)}})_{\ast,\tau_l \tilde v^{(l)}} [ ( \lambda(B^{(l)})  \tilde v ^{(l)}]\big ] \Big ) \right] \\
\end{align*}
where we took into account that
$$\lambda( B^{(l)})  \exp_{x^{(l)}}\tau_l \tilde v^{(l)} = \frac d {dt}  \exp_{x^{(l)}}\big (L_{\e{-t B^{(l)}}} \big )_{\ast,x^{(k)}} \tau_l  \tilde v^{(l)} _{|t=0}=(\exp_{x^{(l)}})_{\ast,\tau_l \tilde v^{(l)}} \big ( \lambda(B^{(l)})  \tau_l \tilde v ^{(l)}\big ).
$$
Since 
\begin{gather*}
\zeta_{kl}^{-1}(\{x^{(k)}\} \times \{ \tau_k\} \times \pi_{kl}^{-1} (x^{(l)})\times \g_{x^{(k)}}\times \g_{x^{(k)}}^\perp )\\=\{x^{(k)}\} \times \{\tau_k\} \times \{x^{(l)}\} \times (-1,1) \times \gamma^{(l)}\big ((S_{kl}^+)_{x^{(l)}}\big )  \times \g _{x^{(l)}}\times \g _{x^{(l)}}^\perp \times \g _{x^{(k)}}^\perp,
\end{gather*}
we obtain for $ I_{kl}(\mu)$ the expression
\begin{align*}
&\int_{M_k(H_k)\times (-1,1)} \Big [ \int_{\gamma^{(k)}((S^+_k)_{x^{(k)}})_l(H_{l})} \Big [ \int_{\zeta_{kl}^{-1}(\{x^{(k)}\}\times \{ \tau_k\} \times \pi_{kl}^{-1} (x^{(l)})\times \g_{x^{(k)}}\times \g_{x^{(k)}}^\perp )\times T^\ast _{m^{(kl)}}W_k} e^{i\frac{\tau_k\tau_l }\mu  \, ^{(kl)} \phw}  \\ &\times  ( (a_k  \circ (\id_{fiber} \otimes \zeta_k )) \chi_{kl} ) \circ \zeta_{kl} \,  \tilde  \Phi_{kl} \, \d (T^\ast _{m^{(kl)}}W_k)(\eta) \, d A^{(k)} \, d A^{(l)} \, dB^{(l)} \,  d\tilde v^{(l)}  \, d\tau_l \Big ] dx^{(l)}  \Big ] \, d\tau_k \,   dx^{(k)}\\
&=\int_{M_k(H_k)\times (-1,1)} \Big [ \int_{\gamma^{(k)}((S^+_k)_{x^{(k)}})_l(H_{l})\times (-1,1)} \Big [ \int_{\gamma^{(l)}((S^+_{kl})_{x^{(l)}}) \times \g_{x^{(l)}} \times \g_{x^{(l)}}^\perp \times \g_{x^{(k)}}^\perp )\times T^\ast _{m^{(kl)}}W_k } e^{i\frac{\tau_k\tau_l }\mu  \, ^{(kl)} \phw}  \\ &\times  ( (a\chi_k   \circ (\id_{fiber} \otimes \zeta_k )) \chi_{kl} ) \circ \zeta_{kl} \,  \tilde  \Phi_{kl} \, \d(T^\ast _{m^{(kl)}}W_k)(\eta)  \, d A^{(k)} \, d A^{(l)} \, dB^{(l)} \,  d\tilde v^{(l)}   \Big ]  d\tau_l \,dx^{(l)}  \Big ] \, d\tau_k \,   dx^{(k)},
\end{align*}
where $m^{(kl)}=\exp_{x^{(k)}} \tau_k \exp_{x^{(l)}} \tau_l \tilde v^{(l)}$, and $d\tilde v^{(l)}$ is a suitable measure in $\gamma^{(l)}((S_{kl}^+)_{x^{(l)}})$ such that we have the equality 
$$\d X \d \eta \equiv \tilde \Phi_{kl} \, \d(T^\ast _{m^{(kl)}}W_k)(\eta)\, d A^{(k)} \, dA^{(l)} \, dB^{(l)} \,  d\tilde v^{(l)}  \, d\tau_l \, dx^{(l)}\, d\tau_k \, dx^{(k)}.$$
Furthermore, $\tilde \Phi_{kl} = |\tau_l | ^{c^{(l)} +d^{(k)} +d^{(l)} -1} \Phi_{kl} \circ \zeta_{kl}$.

\subsection*{Second reduction} Now, the group $G_{x^{(k)}}$ acts on $\gamma^{(l)}((S_k)_{x^{(k)}})_l$ with the isotropy types $(H_l)=(H_{i_j}),(H_{i_{j+1}}), \dots, (H_L)$. By the same arguments given in the first reduction, the isotropy types occuring in $W_{kl}$ constitute a subset of these types, and we shall denote them by
\bqn
(H_l) = (H_{i_{r_1}}), (H_{i_{r_2}}), \dots, (H_L).
\eqn
Consequently, $G_{x^{(k)}}$ acts on $S_{kl}$ with the isotropy types $(H_{i_{r_2}}), \dots, (H_L)$. Again, if $G$ acted on $S_{kl}$ only with type $(H_L)$, we shall see in the next section that the critical set of $^{(kl)} \phw$ would be clean. However, in general this will not be the case, and we have to continue with the iteration.

\subsection*{N-th decomposition}
Once one arrives at a sphere bundle $S_{klmn...}$ on which $G$ acts only with the isotropy type $(H_L)$,  the end of the iteration will be reached. More precisely, let $N\geq 3$,  $(H_{i_1}), \dots , (H_{i_{N+1}})=(H_L)$ be a branch of the isotropy tree of the $G$-action on $M$,  and   $f_{i_1}$, $f_{i_1i_2}$, $S_{i_1}$, $S_{i_1i_2}$, as well as  $x^{(i_1)}\in M_{i_1}(H_{i_1}), \quad x^{(i_2)}\in \gamma^{(i_{1})}\big ((S_{i_1}^+)_{x^{(i_{1})}}\big ) _{i_2}(H_{i_2})$ 
be defined as in the first two iteration steps. Let now $N \geq j \geq 3$, and assume that $f_{i_1\dots i_{j-1}}$,  $S_{i_1\dots i_{j-1}}$,... have already been defined. 
Let $\gamma^{(i_{j-1})}((S_{i_1\dots i_{j-1}})_{x^{(i_{j-1})}})_{i_{j}}$ be the submanifold with corners of $\gamma^{(i_{j-1})}((S_{i_1\dots i_{j-1}})_{x^{(i_{j-1})}})$ from which all the isotropy types  less than $(H_{i_{j}})$ have been removed. Consider  the invariant tubular neighborhood $f_{i_1\dots i_{j}}=\exp \circ \gamma^{(i_{j})}: \nu _{i_1\dots i_{j}} \rightarrow\gamma^{(i_{j-1})}((S_{i_1\dots i_{j-1}})_{x^{(i_{j-1})}})_{i_{j}}$ of the set of maximal singular orbits $\gamma^{(i_{j-1})}((S_{i_1\dots i_{j-1}})_{x^{(i_{j-1})}})_{i_{j}}(H_{i_{j}})$, and define $S_{i_1\dots i_{j}}$  as the sphere subbundle in $ \nu _{i_1\dots i_{j}}$. For  $x^{(i_{j})}\in \gamma^{(i_{j-1})}((S^+_{i_1\dots i_{j-1}})_{x^{(i_{j-1})}})_{i_{j}}(H_{i_{j}})$ we then consider the decomposition
\bqn 
 \g_{x^{(i_{j-1})}}= \g_{x^{(i_j)}}\oplus \g_{x^{(i_j)}}^\perp, 
\eqn
and set $d^{(i_j)}=\dim \g_{p^(i_j)}^\perp$, $ e^{(i_j)}=\dim \g_{p^(i_j)}$. After $N$ iterations, one arrives at the decomposition
\begin{gather*}
\g = \g_{x^{(i_1)}} \oplus \g_{x^{(i_1)}}^\perp =(\g_{x^{(i_2)}}\oplus \g_{x^{(i_2)}}^\perp) \oplus \g_{x^{(i_1)}}^\perp =\dots = \g_{x^{(i_N)}}\oplus \g_{x^{(i_N)}}^\perp \oplus \cdots \oplus \g_{x^{(i_1)}}^\perp, 
\end{gather*}
and we denote by  $\{ A_r^{(i_j)}(x^{(i_1)},\dots,x^{(i_j)})\}$ a basis of $\g_{x^(i_j)}^\perp$, and by $\{ B_r^{(i_N)}(x^{(i_1)},\dots,x^{(i_N)})\}$ a basis of $\g_{x^(i_N)}$. Let further 
\begin{align*}
A^{(i_j)} &=\sum_{r=1}^{d^{(i_j)}} \alpha^{(i_j)}_r A_r^{(i_j)}(x^{(i_1)},\dots,x^{(i_j)}), \qquad B^{(i_N)} =\sum_{r=1}^{e^{(i_N)}} \beta^{(i_N)}_r B_r^{(i_N)}(x^{(i_1)},\dots,x^{(i_N)}),
\end{align*}
and put
\bqn
\tilde v^{(i_N)}(x^{(i_j)},\theta^{(i_N)})= \gamma^{(i_N)} \left ( \Big (v_\rho^{(i_1\dots i_N)}(x^{(i_j)})+ \sum _{r\not=\rho}^{c^{(i_N)}} q_r^{(i_N)} v_r^{(i_1\dots i_N)}(x^{(i_j)})\Big )\Big / \sqrt{1 + \sum\limits_{r\not=\rho} (q_r^{(i_N)})^2} \right ) 
\eqn
for some $\rho$, where $\mklm{v_r^{(i_1\dots i_N)}(x^{(i_1)}, \dots x^{(i_N)}  )}$ is an orthonormal frame in $(\nu_{i_1\dots  i_N})_{x^{(i_N)}}$. 
Finally, we shall use the notations
\begin{align*}
m^{(i_j\dots i_{N})}&=\exp_{x^{(i_j)}}[\tau_{i_j} \exp_{x^(i_{j+1})}[\tau_{i_{j+1}}\exp_{x^(i_{j+2})}[\dots [ \tau_{i_{N-2}}\exp_{x^(i_{N-1})}[\tau_{i_{N-1}}\exp_{x^{(i_N)}}[ \tau_{i_N} \tilde v ^{(i_N)}]]] \dots ]]], \\
X^{(i_j\dots i_{N})}&={\tau_{i_j} \cdots \tau_{i_N}A^{(i_j)}}+{\tau_{i_{j+1}} \cdots \tau_{i_N}A^{(i_{j+1})}}+\dots +{\tau_{i_{N-1}}  \tau_{i_N}A^{(i_{N-1})}} +{\tau_{i_N} A^{(i_N)}}  +B^{(i_N)},
\end{align*}
where $j=1,\dots ,N$. Consider now for every fixed $x^{(i_{N-1})} \in \gamma^{(i_{N-2})}((S_{i_1\dots i_{N-2}})_{x^{(i_{N-2})}})_{i_{N-1}}(H_{i_{N-1}})$ the decomposition of the closed $G_{x^{(i_{N-1})}}$-manifold
 $\gamma^{(i_{N-1})}((S_{i_1\dots i_{N-1}})_{x^{(i_{N-1})}})$ given by 
\begin{gather*}
\gamma^{(i_{N-1})}((S_{i_1\dots i_{N-1}})_{x^{(i_{N-1})}}) =W_{i_1\dots i_{N}} \, \cup \, W_{i_1\dots i_{N-1} L}, \\ W_{i_1 \dots i_N}= f_{i_1\dots i_N}(\stackrel \circ D_1(\nu_{i_1\dots i_N})), \quad W_{i_1 \dots i_{N-1}L}=\mathrm{Int}  (\gamma^{(i_{N-1})}((S_{i_1\dots i_{N-1}})_{x^{(i_{N-1})}})_L),
\end{gather*}
where $f_{i_1\dots i_N}:\nu_{i_1\dots i_N } \rightarrow \gamma^{(i_{N-1})}((S_{i_1\dots i_{N-1}})_{x^{(i_{N-1})}})_{i_N}$ is an invariant tubular neighborhood of the closed invariant submanifold $\gamma^{(i_{N-1})}((S_{i_1\dots i_{N-1}})_{x^{(i_{N-1})}})_{i_N} (H_{i_N})$ in $\gamma^{(i_{N-1})}((S_{i_1\dots i_{N-1}})_{x^{(i_{N-1})}})_{i_N}=\gamma^{(i_{N-1})}((S_{i_1\dots i_{N-1}})_{x^{(i_{N-1})}})$, and 
\bqn 
\gamma^{(i_{N-1})}((S_{i_1\dots i_{N-1}})_{x^{(i_{N-1})}})_L=\gamma^{(i_{N-1})}((S_{i_1\dots i_{N-1}})_{x^{(i_{N-1})}})-f_{i_1\dots i_N}(\stackrel \circ D_{1/2}(\nu_{i_1\dots i_N})).
\eqn
Let $\mklm{\chi_{i_1\dots i_{N}}, \chi_{i_1\dots i_{N-1} L}}$ denote a partition of unity subordinated to the covering by  the open sets $\{W_{i_1\dots i_{N}}, W_{i_1\dots i_{N-1} L}\}$, and decompose $I_{i_1\dots i_{N-1}}(\mu)$ accordingly, so that
\bqn
I_{i_1\dots i_{N-1}}(\mu)= I_{i_1\dots i_N}(\mu) + I_{i_1\dots i_{N-1}L}(\mu) . 
\eqn
 \subsection*{N-th monoidal transformation} In the chart $(-1,1)^{N-1} \times \gamma^{(i_{N-1})}(S^+_{i_1\dots i_{N-1}}) \times \g$ consider the  monoidal transformation 
 \bqn 
 \zeta_{i_1\dots i_N}: B_{Z_{i_1\dots i_N}}((-1,1)^{N-1} \times \gamma^{(i_{N-1})}(S^+_{i_1\dots i_{N-1}})\times \g)\longrightarrow (-1,1)^{N-1} \times \gamma^{(i_{N-1})}(S^+_{i_1\dots i_{N-1}}) \times \g
 \eqn
  with center
 \begin{gather*}
 Z_{i_1\dots i_N}= (-1,1)^{N-1} \times \mathfrak{iso} \,  \Gamma_{i_1\dots i_{N-1}, i_N}^+(H_{i_N}), \\ \Gamma_{i_1\dots i_{N-1}, i_N}=\bigcup _{x^{(i_{N-1})}} \gamma^{(i_{N-1})}((S_{i_1\dots i_{N-1}})_{x^{(i_{N-1})}})_{i_N}=\gamma^{(i_{N-1})}((S_{i_1\dots i_{N-1}})).
 \end{gather*}
 For an arbitrary element $A ^{(i_j)}\in \g_{i_j}^\perp$ one  computes
 \begin{align*}
 (\tilde A ^{i_j)})_{ m^{(i_1\dots i_{N})}}&=\frac d{dt} \e{-t A^{(i_j)}} \cdot  m^{(i_1\dots i_{N})}_{|t=0}= \frac d {dt} \exp_{x^{(i_1)}} \big [(\e{-t A^{(i_j)}})_{\ast, x^{(i_1)}} [\tau_{i_1}  m^{(i_2\dots i_{N})}]\big ]_{|t=0}\\
 &=(\exp_{x^{(i_1)}})_{\ast, \tau_{i_1}  m^{(i_2\dots i_{N})}}[ \lambda(A^{(i_j)})\tau_{i_1}  m^{(i_2\dots i_{N})}],
 \end{align*}
 successively obtaining
  \begin{align*}
 (\tilde A ^{i_j)})_{ m^{(i_1\dots i_{N})}}&= \frac d {dt} \exp_{x^{(i_1)}} \big [\tau_{i_1} \exp_{x^{(i_2)}} [\dots [\tau_{i_{j-1}}  (\e{-t A^{(i_j)}})_{\ast, x^{(i_1)}} m^{(i_j\dots i_{N})}]\dots ]\big ]_{|t=0}\\
 &=(\exp_{x^{(i_1)}})_{\ast, \tau_{i_1}  m^{(i_2\dots i_{N})}} \big [ \tau_{i_1} (\exp_{x^{(i_2)}})_{\ast, \tau_{i_2} m^{(i_3\dots i_N)}}[ \dots [\tau_{i_{j-1}}  \lambda(A^{(i_j)})  m^{(i_j\dots i_{N})}]\dots ]\big ].
 \end{align*}
 As a consequence, the phase function factorizes according to 
\bqn
\, ^{(i_1\dots i_N)} \tilde \psi^{tot}=\mathbb{J}( \eta_{m^{(i_1 \dots i_N)}})( X^{(i_1\dots i_N)})=\tau_{i_1} \cdots \tau_{i_N} \, ^{(i_1\dots i_N)}\tilde \psi^ {wk},
\eqn
where $\eta_{m^{(i_1 \dots i_N)}}  \in \pi^{-1} (m^{(i_1\dots i_N)})$, and  
\begin{align*}
\, ^{(i_1\dots i_N)}\tilde \psi^ {wk}&= \eta_{m^{(i_1 \dots i_N)}} \Big (\widetilde{ A^{(i_1)}}_{m^{(i_1 \dots i_N)}}\Big ) + \sum_{j=2}^N  \eta_{m^{(i_1 \dots i_N)}} \Big ( (\exp_{x^{(i_1)}})_{\ast, \tau_{i_1}  m^{(i_2\dots i_N)} } \\
 & \big [ (\exp_{x^{(i_2)}})_{\ast,\tau_{i_2}  m^{(i_3\dots i_N)}} \big [ \dots (\exp_{x^{(i_{j-1})}})_{\ast, \tau_{i_{j-1}}m^{(i_j\dots i_N)}}[\lambda(A^{(i_j)})  m ^{(i_j\dots i_N)}] \dots \big ]
\big ] \Big ) \\
& +  \eta_{m^{(i_1 \dots i_N)}} \Big ( (\exp_{x^{(i_1)}})_{\ast, \tau_{i_1}  m^{(i_2\dots i_N)} }\big [ (\exp_{x^{(i_2)}})_{\ast,\tau_{i_2}  m^{(i_3\dots i_N)}} \big [ \dots \\
&(\exp_{x^{(i_{N})}})_{\ast, \tau_{i_{N}}\tilde v^{(i_N)}} [\lambda(B^{(i_N)})  \tilde v ^{(i_N)}] \dots \big ]
\big ] \Big )
\end{align*}
in the given charts. With $S_{i_1 \dots i_{N}}$ equal to the sphere bundle over $\gamma^{(i_{N-1})}((S_{i_1\dots i_{N-1}})_{x^{(i_{N-1})}})_{i_N} (H_{i_N})$, one finally obtains for the integral $I_{i_1\dots i_N}(\mu)$ the expression
\begin{gather}
\label{eq:N}
\begin{split}
I_{i_1\dots i_N}(\mu)\qquad \qquad \qquad \qquad\qquad \qquad\qquad \qquad\qquad \qquad \\
 =\int_{M_{i_1}(H_{i_1})\times (-1,1)} \Big [ \int_{\gamma^{(i_1)}((S^+_{i_1})_{x^{(i_1)}})_{i_2}(H_{i_2})\times (-1,1)} \dots \Big [  \int_{\gamma^{(i_{N-1})}((S^+_{i_1\dots i_{N-1}})_{x^{(i_{N-1})}})_{i_{N}}(H_{i_{N}})\times (-1,1)} \\
\Big [ \int_{\gamma^{(i_{N})}((S_{i_1\dots i_{N}}^+)_{x^{(i_N)}})\times \g_{x^{(i_{N})}}\times \g_{x^{(i_{N})}}^\perp \times \cdots \times \g_{x^{(i_{1})}}^\perp\times T^\ast _{m^{(i_1\dots i_N)}}W_{i_1}}  e^{i\frac {\tau_1 \dots \tau_N}\mu \, ^{(i_1\dots i_N)} \tilde \psi ^{wk}}  \, a_{i_1\dots i_N} \,   \tilde \Phi_{i_1\dots i_N} \\
  \d(T^\ast _{m^{(i_1\dots i_N)}}W_{i_1})(\eta)  \d A^{(i_1)} \dots  \d A^{(i_N)}  \d B^{(i_N)} \d \tilde v^{(i_N)} \Big ]  \d \tau_{i_N} \d x^{(i_{N})} \dots  \Big ] \d \tau_{i_2} \d x^{(i_{2})} \Big ]\d \tau_{i_1} \d x^{(i_{1})}.
\end{split}
\end{gather}
Here
\bqn 
a_{i_1\dots i_N}=[ a \, \chi_{i_1}  \circ (\id_{fiber} \otimes \zeta_{i_1} \circ \zeta_{i_1i_2} \circ \dots \circ \zeta_{i_1\dots i_N})] \, [ \chi_{i_1i_2} \circ \zeta_{i_1i_2} \circ \dots\circ \zeta_{i_1\dots i_N}  ] \dots [\chi_{i_1\dots i_N} \circ \zeta_{i_1\dots i_N}]
\eqn
is supposed to have compact support in one of the $\theta^{(i_N)}$-charts, and
\begin{align*}
\tilde \Phi_{i_1\dots i_N} &=\prod_{j=1}^N |\tau_{i_j}|^{c^{(i_j)}+\sum_{r=1}^j d^{(i_r)}-1}\Phi_{i_1\dots i_N},
\end{align*}
where $\Phi_{i_1\dots i_N}$ is a smooth function which does not depend on the variables $\tau_{i_j}$.

\subsection*{N-th reduction} By assumption, $G$ acts on $S_{i_1\dots i_N}$ only with type $(H_L)$, and the iteration process ends here.

\section{Phase analysis of the weak transform. The first fundamental theorem}

We are now in position to state the first fundamental theorem in  the derivation of equivariant spectral asymptotics. For this end, let us  define certain geometric distributions $E^{(i_j)}$ and $F^{(i_N)}$ on $M$ associated to the iteration of $N$ steps along the branch $((H_{i_1}), \dots , (H_{i_{N+1}})=(H_L))$ of the isotropy tree of the $G$-action on $M$ by setting
\begin{equation}
\label{eq:EF}
\begin{split}
E^{(i_1)}_{m^{(i_1 \dots i_N)}}&=\mathrm{Span} \{  \tilde Y _{m^{(i_1 \dots i_N)}}: Y \in \g_{x^{(i_1)}} ^\perp\}, \\ 
E^{(i_j)} _{m^{(i_1 \dots i_N)}}&= (\exp_{x^{(i_1)}})_{\ast, \tau_{i_1}  m^{(i_2\dots i_N)} } \dots  (\exp_{x^{(i_{j-1})}})_{\ast, \tau_{i_{j-1}}m^{(i_j\dots i_N)}}[\lambda(\g_{x^{(i_j)}}^\perp)  m ^{(i_j\dots i_N)}]
, \\
F^{(i_N)}_{m^{(i_1 \dots i_N)}} &=  (\exp_{x^{(i_1)}})_{\ast, \tau_{i_1}  m^{(i_2\dots i_N)} } \dots (\exp_{x^{(i_{N})}})_{\ast, \tau_{i_{N}}\tilde v^{(i_N)}}[\lambda(\g_{x^{(i_N)}})  \tilde v ^{(i_N)}],
\end{split}
\end{equation}
where $2 \leq j \leq N$, the notation being as in the previous section. By construction, for $\tau_{i_j}\not=0$, $1\leq j\leq N$, the $G$-orbit through $m^{(i_1\dots i_N)}$ is of principal type $G/H_L$, which amounts to the fact that $G$ acts on $S_{i_1\dots i_N}$ only with the isotropy type $(H_L)$.  Let $\eta_{m^{(i_1 \dots i_N)}}  \in \pi^{-1} (m^{(i_1\dots i_N)})$. We then have the following

\begin{theorem}%[1. Fundamental Theorem]
Consider the factorization 
\bqn
\mathbb{J}( \eta_{m^{(i_1 \dots i_N)}})( X^{(i_1\dots i_N)})= \, ^{(i_1\dots i_N)} \tilde \psi^{tot}=
\tau_{i_1} \cdots \tau_{i_N} \, ^{(i_1\dots i_N)}\tilde \psi^ {wk, \, pre}
\eqn
of the phase function $\psi$ after $N$ iteration steps, where $\, ^{(i_1\dots i_N)}\tilde \psi^ {wk,pre}$ is given by
\begin{align*}
&\eta_{m^{(i_1 \dots i_N)}} \Big (\widetilde{ A^{(i_1)}}_{m^{(i_1 \dots i_N)}}\Big ) + \sum_{j=2}^N  \eta_{m^{(i_1 \dots i_N)}} \Big ( (\exp_{x^{(i_1)}})_{\ast, \tau_{i_1}  m^{(i_2\dots i_N)} }\big [ (\exp_{x^{(i_2)}})_{\ast,\tau_{i_2}  m^{(i_3\dots i_N)}} \big [ \dots\\
& (\exp_{x^{(i_{j-1})}})_{\ast, \tau_{i_{j-1}}m^{(i_j\dots i_N)}}[\lambda(A^{(i_j)})  m ^{(i_j\dots i_N)}] \dots \big ]
\big ] \Big )  +  \eta_{m^{(i_1 \dots i_N)}} \Big ( (\exp_{m^{(i_1)}})_{\ast, \tau_{i_1}  m^{(i_2\dots i_N)} }  \\ 
& \big [ (\exp_{x^{(i_2)}})_{\ast,\tau_{i_2}  m^{(i_3\dots i_N)}} \big [ \dots  (\exp_{x^{(i_{N})}})_{\ast, \tau_{i_{N}}\tilde v^{(i_N)}} [\lambda(B^{(i_N)})  \tilde v ^{(i_N)}] \dots \big ] \big ] \Big ), 
\end{align*}
 Let further 
\bqn 
 \, ^{(i_1\dots i_N)}\tilde \psi^ {wk}
\eqn
denote the pullback of  $ \, ^{(i_1\dots i_N)}\tilde \psi^ {wk,\,pre}$ along the  substitution $\tau=\delta_{i_1\dots i_N}(\sigma)$ given by the sequence of monoidal transformations
\begin{align*}
\delta_{i_1\dots i_N}: (\sigma_{i_1}, \dots \sigma_{i_N}) &\mapsto \sigma_{i_1}( 1, \sigma_{i_2}, \dots, \sigma_{i_N})= (\sigma_{i_1}', \dots ,\sigma_{i_N}')\mapsto \sigma_{i_2}'(\sigma_{i_1}',1,\dots, \sigma_{i_N}')= (\sigma_{i_1}'', \dots, \sigma_{i_N}'')\\
 &\mapsto \sigma_{i_3}''(\sigma_{i_1}'',\sigma_{i_2}'', 1,\dots, \sigma_{i_N}'')= \cdots \mapsto \dots = (\tau_{i_1}, \dots ,\tau_{i_N}).
\end{align*}
Then the critical set $\Crit(\, ^{(i_1\dots i_N)} \phw)$ of $\, ^{(i_1\dots i_N)} \phw$ is given by all points
$$(\sigma_{i_1}, \dots, \sigma_{i_N}, x^{(i_1)}, \dots, x^{(i_N)},  \tilde v ^{(i_N)}, A^{(i_1)}, \dots, A^{(i_N)}, B^{(i_N)}, \eta_{m^{(i_1 \dots i_N)}}) $$
satisfying the conditions

\medskip
\begin{tabular}{ll}
\emph{(I)} &  $A^{(i_j)} =0$ for all $j=1,\dots,N$, and $\lambda(  B^{(i_N)})\tilde v^{(i_N)}= 0$; \\[2pt]
\emph{(II)} &  $\eta_{m^{(i_1 \dots i_N)}} \in \mathrm{Ann}\big (E^{(i_j)} _{m^{(i_1 \dots i_N)}}\big )$ for all $j=1,\dots, N$; \\[2pt]
\emph{(III)} &   $ \eta_{m^{(i_1 \dots i_N)}}\in\mathrm{Ann} \big (  F^{(i_N)} _{m^{(i_1 \dots i_N)}}\big )$.
\end{tabular}
\medskip

\noindent
Furthermore,  $\Crit(\, ^{(i_1\dots i_N)} \phw)$ is a $\Cinft$-submanifold of codimension $2\kappa$, where $\kappa=\dim G/H_L$ is the dimension of a principal orbit.
\end{theorem}
\begin{proof}
 To begin with, let  $\sigma_{i_1} \cdots \sigma_{i_N}\not=0$. In this case, the sequence of monoidal transformations $\zeta=\zeta_{i_1} \circ \zeta_{i_1i_2} \circ \dots \circ \zeta_{i_1\dots i_N}\circ \delta_{i_1\dots i_N}$ constitutes a diffeomorphism, so that 
 \begin{align*}
 \mathrm{Crit}(\, ^{(i_1\dots i_N)} \psi^{tot})_{\sigma_{i_1} \cdots \sigma_{i_N}\not=0}=\{&(\sigma_{i_1}, \dots, \sigma_{i_N}, x^{(i_1)}, \dots, x^{(i_N)},  \tilde v ^{(i_N)}, A^{(i_1)}, \dots, A^{(i_N)}, B^{(i_N)}, \eta_{m^{(i_1 \dots i_N)}})  \\ & \in \tilde {\mathcal{C}}^{tot}, \quad  {\sigma_{i_1} \cdots \sigma_{i_N}\not=0}  \},
  \end{align*}
   where $\tilde {\mathcal{C}}^{tot}=((\zeta \otimes \id_{fiber})^{-1}(\Crit{(\psi)})$  denotes the total transform of the critical set of $\psi$. Now, 
 \bqn
 (\eta_{m^{(i_1 \dots i_N)}},X^{(i_1\dots i_N)})\in \mathrm{Crit}(\psi) \quad \Leftrightarrow \quad  \eta_{m^{(i_1 \dots i_N)}} \in \Omega,  \quad \tilde X^{(i_1\dots i_N)}_{\eta_{m^{(i_1 \dots i_N)}}}=0.
 \eqn
Furthermore, $\tilde X_\eta=0$ clearly implies $\tilde X_{\pi(\eta)} =\pi_{\ast}(\tilde X_\eta)=0$. 
Since the point $m^{(i_1\dots i_N)}$ lies in a slice at $x^{(i_1)}$, the condition $\tilde X^{(i_1\dots i_N)}_{m^{(i_1\dots i_N)}}=0$ means that 
the vector field  $\tilde X^{(i_1\dots i_N)}$ must vanish at $x^{(i_1)}$ as well. But
\bqn
\g_m=\mathrm{Lie}{(G_m}) =\mklm{X \in \g: \tilde X_m=0},
\eqn
so that $X^{(i_1\dots i_N)} \in \g_{x^{(i_1)}}$. Next
\begin{align*}
\g_{x^{(i_N)}} \subset \g_{x^{(i_{N-1})}} \subset \dots \subset \g_{x^{(i_1)}}
\end{align*}
and $\g_{x^{(i_{j+1})}}^\perp \subset \g_{x^{(i_j)}}$ imply
\bqn
\tilde X^{(i_1\dots i_N)}_{x^{(i_1)}}= {\tau_{i_1} \dots \tau_{i_N}\sum \alpha_r^{(i_1)} (\tilde A_r ^{(i_1)}})_{ x^{(i_1)}} =0.
\eqn
Thus we conclude $\alpha^{(i_1)} =0$, which gives $X^{(i_2\dots i_N)} \in \g_{m^{(i_1\dots i_N)}}$, and consequently $X^{(i_2\dots i_N)} \in \g_{m^{(i_2\dots i_N)}}$. Repeating the above argument we actually obtain for $\sigma_{i_j}\not=0$
\bq
\label{eq:G}
\g_{m^{(i_1 \dots i_N)}} = \g _{\tilde v ^{(i_N)}},
\eq
since $\g_{\tilde v ^{(i_N)}} \subset \g_{x ^{(i_N)}}$. Therefore  the condition $\tilde X^{(i_1\dots i_N)}_{m^{(i_1\dots i_N)}}=0$ is equivalent to (I) in the case that all $\sigma_{i_j}$ are different from zero. Now, $ \eta_{m^{(i_1 \dots i_N)}} \in \Omega$ means that 
\bqn
\mathbb{J}(\eta_{m^{(i_1 \dots i_N)}})(X)=\eta_{m^{(i_1 \dots i_N)}} (\tilde X_{m^{(i_1 \dots i_N)}}) =0 \qquad \forall X \in \g,
\eqn
which is equivalent to  $ \eta_{m^{(i_1\dots i_ N)}} \in \mathrm{Ann}(T_{m^{(i_1\dots i_ N)}} (G \cdot {m^{(i_1\dots i_ N)}}))$. If $\sigma_{i_j}\not=0$ for all $j=1,\dots,N$, (II) and (III) imply that 
\bqn 
\eta_{m^{(i_1\dots i_ N)}}\Big ( (\exp_{x^{(i_1)}})_{\ast, \tau_{i_1}  m^{(i_2\dots i_N)} } [\dots  (\exp_{x^{(i_{j-1})}})_{\ast, \tau_{i_{N-1}}m^{(i_N)}}[\lambda(Z)  m ^{(i_N)}] \dots \big ] \Big )=0 \quad \forall Z \in \g_{x^{(i_{N-1})}},
\eqn
since $ \g_{x^{(i_{N-1})}} =  \g_{x^{(i_{N})}}\oplus  \g_{x^{(i_{N})}}^\perp$. By repeatedly using this argument, we  conclude  that for $\sigma_{i_j}\not= 0$
\bq
\label{eq:IVb}
\mathrm{(II), \, (III)} \quad \Leftrightarrow \quad \eta_{m^{(i_1 \dots i_N)}} \in \mathrm{Ann}(T_{m^{(i_1\dots i_ N)}} (G \cdot {m^{(i_1\dots i_ N)}})).
\eq
Taking everything together therefore gives
\begin{align}
\begin{split}
\label{eq:XX}
\mathrm{Crit}(\, ^{(i_1\dots i_N)}& \psi^{tot})_{\sigma_{i_1}\cdots \sigma_{i_N}\not=0}\\ 
 &=\{(\sigma_{i_1}, \dots, \sigma_{i_N}, x^{(i_1)}, \dots, x^{(i_N)},  \tilde v ^{(i_N)}, A^{(i_1)}, \dots, A^{(i_N)}, B^{(i_N)}, \eta_{m^{(i_1 \dots i_N)}}):  \\    &{\sigma_{i_1} \cdots \sigma_{i_N}\not=0}, \, \text{(I)-(III) are fulfilled and $ \tilde B^{(i_N),\mathrm{v}}_{\eta_{m^{(i_1 \dots i_N)}}}=0$}   \}.
\end{split}
  \end{align}  
  Here $\mathfrak{X}_\eta^{\mathrm{v}}$ denotes the vertical component of  a vector field $\mathfrak{X} \in T(T^\ast M)$ with respect to the decomposition $T_\eta(T^\ast M)=T^{\mathrm{v}} \oplus T^{\mathrm{h}}$, $T^{\mathrm{v}}$ being  the tangent space to the fiber $T_\eta^\ast M$ at zero, and $T^{\mathrm{h}}$ the tangent space to the zero section $M \subset T^\ast M$ at $\eta$. We now assert that   
$$ \mathrm{Crit}(\, ^{(i_1\dots i_N)} \psi^{wk})=\overline{\mathrm{Crit}(\, ^{(i_1\dots i_N)} \psi^{tot})_{\sigma_{i_1} \cdots \sigma_{i_N}\not=0}}.$$
To show this, let us write  $\eta_{m^{(i_1 \dots i_N)}} =\sum p_i \d q_i$ with respect to some local coordinates $q_1,\dots, q_n$, and still assume that all $\sigma_{i_j}$ are different from zero. Then all $\tau_{i_j}$ are different from zero, too, and $\gd_p \, ^{(i_1\dots i_N)} \phw=0$ is equivalent to
\begin{gather*}
 \gd _p \mathbb{J}(\eta_{m^{(i_1 \dots i_N)}})( X^{(i_1\dots i_N)})= ( \d q_1 (\tilde X^{(i_1\dots i_N)}_{m^{(i_1\dots i_N)}}), \dots, \d q_n (\tilde X^{(i_1\dots i_N)}_{m^{(i_1\dots i_N)}}))=0,
\end{gather*}
which gives us the condition $\tilde X^{(i_1\dots i_N)}_{m^{(i_1\dots i_N)}}=0$. By \eqref{eq:G} we therefore obtain condition I) in the case that all $\sigma_{i_j}$ are different from zero. Let now one of the  $\sigma_{i_j}$ be equal to zero, so that  all $\tau_{i_j}$ are zero. With the identification $T_0(T_mM) \simeq T_mM$  one has 
\bqn 
(\exp_m)_{\ast,0}: T_0(T_mM) \longrightarrow T_mM, \qquad  (\exp_m)_{\ast,0}\simeq \id,
\eqn
and similarly $(\exp_{x^{(i_j)}})_{\ast,0}\simeq \id$ for all $j=2,\dots, N$, so that
\begin{align}
\label{eq:B}
\, ^{(i_1\dots i_N)}\tilde \psi^ {wk}&= \sum p_i \, dq_i \Big (\widetilde{ A^{(i_1)}}_{x^{(i_1)}} + \sum_{j=2}^N  \lambda(A^{(i_j)})  x ^{(i_j)} + \lambda(B^{(i_N)})  \tilde v ^{(i_N)}  \Big ).
\end{align}
Therefore $\gd_p \, ^{(i_1\dots i_N)} \phw=0$ is equivalent to
\bqn
\widetilde{ A^{(i_1)}}_{x^{(i_1)}} + \sum_{j=2}^N  \lambda(A^{(i_j)})  x^{(i_j)} + \lambda(B^{(i_N)})  \tilde v ^{(i_N)}=0.
\eqn
Now,  let $N_{x^{(i_1)}} ( G \cdot x^{(i_1)})$ be the normal space in $T_{x^{(i_1)}}M$ to the orbit $G \cdot x^{(i_1)}$, on which $G_{x^{(i_1)}}$ acts, and define  $N_{x^{(i_{j+1})}} ( G_{x^{(i_{j})}} \cdot x^{(i_{j+1})})$ successively as the normal space to the orbit $ G_{x^{(i_{j})}} \cdot x^{(i_{j+1})}$  in the $ G_{x^{(i_{j})}}$-space $N_{x^{(i_j)}} ( G_{x^{(i_{j-1})}} \cdot x^{(i_j)})$, where we understand that $G_{x^{(i_0)}}=G$. By Bredon \cite{bredon}, page 308, these actions can be assumed to be orthogonal. Set  
\bq
\label{eq:V}
V^{(i_1\dots i_{j})}= \bigcap_{r=1} ^{j} N_{x^{(i_r)}} ( G_{x^{(i_{r-1})}} \cdot x^{(i_r)})= N_{x^{(i_j)}} ( G_{x^{(i_{j-1})}} \cdot x^{(i_j)}).
\eq
Since $x^{(i_j)} \in \gamma^{(i_{j-1})} (S^+_{i_1\dots i_{j-1}})_{x^{(i_{j-1})}})\subset V^{(i_1\dots i_{j-1})}$, we see that for every $j=2,\dots, N$
\bqn 
\lambda\Big (\sum_r \alpha_r^{(i_j)} A_r^{(i_j)} \Big ) \, x^{(i_j)} \in   T_{x^{(i_j)}} ( G_{x^{(i_{j-1})}} \cdot x^{(i_j)})\subset V^{(i_1\dots i_{j-1})} .
\eqn 
In addition, $(\tilde  A_r^{(i_1)})_{x^{(i_1)}}\in   T_{x^{(i_1)}} ( G \cdot x^{(i_1)})$, and
$
\lambda\Big (\sum_r \beta^{(i_N)} _r B_r^{(i_N)}\Big ) \tilde v ^{(i_N)} \in V^{(i_1\dots i_{N})}
$, so that taking everything together we obtain for arbitrary $\sigma_{i_j}$
\bqn
\gd_p \, ^{(i_1\dots i_N)} \phw=0 \quad \Longleftrightarrow \quad \mathrm{(I)}.
\eqn
 In particular, one concludes that $\, ^{(i_1\dots i_N)} \phw$ must vanish on its critical set. Since 
\bqn 
d(\, ^{(i_1\dots i_N)} \psi^{tot})= d(\tau_{i_1}\dots \tau_{i_N}) \cdot \, ^{(i_1\dots i_N)} \psi^{wk} +  \tau_{i_1}\dots \tau_{i_N} d\,( ^{(i_1\dots i_N)} \psi^{wk}),
\eqn
one sees that 
\bqn 
 \mathrm{Crit}(\, ^{(i_1\dots i_N)} \psi^{wk})\subset \mathrm{Crit}(\, ^{(i_1\dots i_N)} \psi^{tot}).
\eqn
In turn, the vanishing of $\psi$ on its critical set implies
\bqn 
 \mathrm{Crit}(\, ^{(i_1\dots i_N)} \psi^{wk})_{\sigma_{i_1}\dots \sigma_{i_N}\not=0}=  \mathrm{Crit}(\, ^{(i_1\dots i_N)} \psi^{tot})_{\sigma_{i_1}\dots \sigma_{i_N}\not=0}.
\eqn
Therefore, by continuity, 
\bq
\label{eq:XXbis}
\overline{\mathrm{Crit}(\, ^{(i_1\dots i_N)} \psi^{tot})_{\sigma_{i_1}\dots \sigma_{i_N}\not=0}} \subset  \mathrm{Crit}(\, ^{(i_1\dots i_N)} \psi^{wk}).
\eq
In order to see the converse inclusion, let us consider next the $\alpha$-derivatives. Clearly,  
\begin{align*}
\gd_{\alpha^{(i_1)}} \, ^{(i_1\dots i_N)} \phw=0 \quad & \Longleftrightarrow \quad \eta_{m^{(i_1 \dots i_N)}}(\tilde Y _{m^{(i_1 \dots i_N)}})=0 \quad \forall \, Y \in \g_{x^{(i_1)}}^\perp.
\end{align*}
For the remaining  derivatives one computes 
\begin{gather*}
\gd_{\alpha_r^{(i_j)}} \, ^{(i_1\dots i_N)} \phw \\ =\eta_{m^{(i_1 \dots i_N)}} \Big ( (\exp_{x^{(i_1)}})_{\ast, \tau_{i_1}  m^{(i_2\dots i_N)} }\big [ \dots (\exp_{x^{(i_{j-1})}})_{\ast, \tau_{i_{j-1}}m^{(i_j\dots i_N)}}[\lambda(A^{(i_j)}_r)  m ^{(i_j\dots i_N)}] \dots \big ] \Big ), 
\end{gather*}
from which one deduces  that for $j=2,\dots, N$
\begin{align*}
\gd_{\alpha^{(i_j)}} \, ^{(i_1\dots i_N)} \phw=0 \quad  \Longleftrightarrow   \quad \forall \, Y \in &\g_{x^{(i_j)}}^\perp \\ 
\eta_{m^{(i_1 \dots i_N)}} \Big ( (\exp_{x^{(i_1)}})_{\ast, \tau_{i_1}  m^{(i_2\dots i_N)} } \big [\dots  (\exp_{x^{(i_{j-1})}})_{\ast, \tau_{i_{j-1}}m^{(i_j\dots i_N)}}&[\lambda(Y)  m ^{(i_j\dots i_N)}] \dots \big ] \Big )=0.
\end{align*}
In a similar way, it is not difficult to see that 
\begin{align*}
\gd_{\beta^{(i_j)}} \, ^{(i_1\dots i_N)} \phw=0 \quad  \Longleftrightarrow   \quad \forall \, Z \in &\g_{x^{(i_N)}} \\ 
\eta_{m^{(i_1 \dots i_N)}} \Big ( (\exp_{x^{(i_1)}})_{\ast, \tau_{i_1}  m^{(i_2\dots i_N)} } \big [\dots  (\exp_{x^{(i_{N})}})_{\ast, \tau_{i_{N}}\tilde v^{(i_N)}}&[\lambda(Z)  \tilde v ^{(i_N)}] \dots \big ] \Big )=0.
\end{align*}
by which the necessity of the conditions (I)--(III) is established. In order to see their sufficiency, let them be fulfilled, and assume again that $\sigma_{i_j}\not=0$ for all $j=1,\dots,N$. Then \eqref{eq:IVb} implies that  $ \eta_{m^{(i_1 \dots i_N)}} \in \mathrm{Ann}(T_{m^{(i_1\dots i_ N)}} (G \cdot {m^{(i_1\dots i_ N)}}))$. Now, if $\sigma_{i_j}\not=0$, $G  \cdot m^{(i_1\dots i_N)}$ is of principal type $G/ H_L$ in $M$, so that the isotropy group of $m^{(i_1\dots i_N)}$ must act trivially on $N_{m^{(i_1\dots i_N)}}(G\cdot m^{(i_1\dots i_N)})$, compare Bredon \cite{bredon}, page 181. If therefore $ \mathfrak{X}= \mathfrak{X}_T+ \mathfrak{X}_N$ denotes an arbitrary element in $ T_{m^{(i_1\dots i_N)}}M = T_{m^{(i_1\dots i_ N)}} (G \cdot {m^{(i_1\dots i_ N)}}))\oplus N_{m^{(i_1\dots i_ N)}} (G \cdot {m^{(i_1\dots i_ N)}}))$, and $g \in G_{m^{(i_1\dots i_N)}}$, one computes 
\begin{align*} 
g \cdot \eta_{m^{(i_1\dots i_N)}}( \mathfrak{X})&= [ ( L_{g^{-1}})^\ast _{gm^{(i_1\dots i_N)}} \eta_{m^{(i_1\dots i_N)}}]( \mathfrak{X})= \eta_{m^{(i_1\dots i_N)}} ( ( L_{g^{-1}}) _{\ast, m^{(i_1\dots i_N)}}( \mathfrak{X}_N))\\
&= \eta_{m^{(i_1\dots i_N)}} (  \mathfrak{X}_N)= \eta_{m^{(i_1\dots i_N)}} (  \mathfrak{X}).
\end{align*}
In view of \eqref{eq:G}, and $\lambda( B^{(i_N)})\tilde v^{(i_N)}=0$ we therefore get the condition $\tilde B^{(i_N), \mathrm{v}}_{\eta_{m^{(i_1\dots i_N)}}}=0$.  Let us now assume that one of the $\sigma_{i_j}$ equals zero. Then  
\begin{align}
\label{eq:VII}
\mathrm{(II), \, (III)} \quad \Leftrightarrow \quad & \left \{
\begin{array}{l}
\eta_{x^{(i_1)}} \in \mathrm{Ann}(T_{x^{(i_j)}} (G_{x^{(i_{j-1})}} \cdot {x^{(i_j)}}))  \quad \forall \, j=1, \dots, N, \\
\eta_{x^{(i_1)}} \in \mathrm{Ann}(T_{\tilde v^{(i_N)}} (G_{x^{(i_{N})}} \cdot {\tilde v^{(i_N)}})).
\end{array} \right.
\end{align}
\begin{lemma}
The orbit of the point $\tilde v^{(i_N)}$ in the $G_{x^{(i_N)}}$-space $V^{(i_1\dots i_N)}$ is of principal type.
\end{lemma}
\begin{proof}[Proof of the lemma]
By assumption, for  $\sigma_{i_j}\not=0$,  $1 \leq j \leq  N$, the $G$-orbit of $m^{(i_1\dots i_N)}$ is of principal type $G/ H_L$ in $M$. The theory of compact group actions then implies that this is equivalent to the fact that $m^{(i_2 \dots i_N)} \in V^{(i_1)}$ is of principal type in the $G_{x^{(i_1)}}$-space $V^{(i_1)}$, see Bredon \cite{bredon}, page 181, which in turn is equivalent to the fact that $m^{(i_3 \dots i_N)} \in V^{(i_1i_2)}$ is of principal type in the $G_{x^{(i_2)}}$-space $V^{(i_1i_2)}$, and so forth. Thus, $m^{(i_j \dots i_N)} \in V^{(i_1 \dots i_{j-1})}$ must be of principal type in the $G_{x^{(i_{j-1})}}$-space $V^{(i_1\dots i_{j-1})}$ for all $j=1,\dots N$, and the assertion follows.
\end{proof}
As a consequence of the previous lemma,  the stabilizer of $\tilde v^{(i_N)} $  must act trivially on $N_{\tilde v^{(i_N)}} ( G_{x^{(i_N)}} \cdot\tilde v^{(i_N)} )$. If therefore  $ \mathfrak{X}= \mathfrak{X}_{T}+  \mathfrak{X}_N$ denotes an arbitrary element in 
\begin{align*}
T_{x^{(i_1)}} M &= T_{x^{(i_1)}} (G \cdot {x^{(i_1)}})\oplus N_{x^{(i_1)}} (G \cdot {x^{(i_1)}})
\\& = \bigoplus_{j=1}^N T_{x^{(i_j)}} (G_{x^{(i_{j-1})}} \cdot {x^{(i_j)}}) \oplus  T_{\tilde v^{(i_N)}} (G_{x^{(i_{N})}}\cdot \tilde v^{(i_N)}) \oplus  N_{\tilde v^{(i_N)}} (G_{x^{(i_{N})}}\cdot \tilde v^{(i_N)}),
\end{align*}
 we obtain with \eqref{eq:VII} 
\begin{align*} 
g \cdot \eta_{x^{(i_1)}}( \mathfrak{X})&= [ ( L_{g^{-1}})^\ast _{gx^{(i_1)}} \eta_{x^{(i_1)}}]( \mathfrak{X})= \eta_{x^{(i_1)}} ( ( L_{g^{-1}}) _{\ast, x^{(i_1)}}( \mathfrak{X}_N))\\
&= \eta_{x^{(i_1)}} (  \mathfrak{X}_N)= \eta_{x^{(i_1)}} (  \mathfrak{X}), \qquad  g \in G_{\tilde v^{(i_N)}}.
\end{align*}
Collecting everything together we have shown for arbitrary $\sigma_{i_j}$ that 
\begin{align}
\label{eq:IV}
 \gd_{p, \alpha^{(i_1)}, \dots , \alpha^{(i_N)} , \beta^{(i_N)}} \, ^{(i_1\dots i_N)} \phw=0 \quad \Longleftrightarrow \quad \mathrm{(I), \,(II), \,(III)} \quad & \Longrightarrow \quad \tilde B^{(i_N),\mathrm{v}}_{\eta_{m^{(i_1\dots i_N)}}}=0.
\end{align}
By \eqref{eq:XX} and \eqref{eq:XXbis} we therefore conclude 
\bq
\label{eq:CC}
\overline{\mathrm{Crit}(\, ^{(i_1\dots i_N)} \psi^{tot})_{\sigma_{i_1}\dots \sigma_{i_N}\not=0}} =  \mathrm{Crit}(\, ^{(i_1\dots i_N)} \psi^{wk}).
\eq
Thus we have computed the critical set of $\, ^{(i_1\dots i_N)} \phw$, and it remains to show that it is a $\Cinft$-submanifold of codimension $2\kappa$. For this end, let us note that if  $\sigma_{i_j}=0$ for some $j$, then $E^{(i_1)}_{x^{(i_1)}}  = T_{x^{(i_1)}}(G\cdot x^{(i_1)})$, and  
\bqn 
E^{(i_j)}_{x^{(i_1)}} \equiv T_{x ^{(i_j)}}(G_{x^{(i_{j-1})}} \cdot x ^{(i_j)}) \subset  V^{(i_{1} \dots i_{j-1})}, \qquad 2 \leq j \leq N, 
\eqn
while $F^{(i_N)}_{x^{(i_1)}}\equiv T_{\tilde v^{(i_N)}} (G_{x^{(i_{N})}}\cdot \tilde v^{(i_N)})  \subset V^{(i_1 \dots i_{N})}$.  Therefore $E^{(i_j)}_{x^{(i_1)}} \cap V^{(i_1 \dots i_{j})}=\mklm{0}$, so that we obtain the direct sum of vector spaces
\bqn 
E_{x^{(i_1)}}^{(i_1)}\oplus E_{x^{(i_1)}}^{(i_2)}\oplus \dots  \oplus E_{x^{(i_1)}}^{(i_N)}\oplus F_{x^{(i_1)}}^{(i_N)} \subset T_{x^{(i_1)}}M.
\eqn 
On the other hand, note that if $\sigma_{i_1}\cdots \sigma_{i_N}\not=0$ one has
\begin{align*}
T_{m^{(i_1 \dots i_N)}} ( G\cdot m^{(i_1 \dots i_N)})=E^{(i_1)}_{m^{(i_1 \dots i_N)}}\oplus  \bigoplus _{j=2}^N \tau_{i_1}\dots \tau_{i_{j-1}} E^{(i_j)} _{m^{(i_1 \dots i_N)}}  \oplus  \tau_{i_1}\dots \tau_{i_N} F^{(i_N)}_{m^{(i_1 \dots i_N)}} 
\end{align*}
for dimensional reasons, so that we obtain the direct sum of geometric distributions $\sum_{j=1}^N E^{(i_j)} \oplus F^{(i_N)}$. 
Consequently, we arrive at the characterization
\begin{align}
\label{eq:C}
\begin{split}
&\Crit(\, ^{(i_1\dots i_N)} \phw)\\ 
= \Big \{ A^{(i_j)}=0, \quad \lambda(B^{(i_N)}) \tilde v ^{(i_N)}&=0, \quad  \eta_{m^{(i_1\dots i_N)}}\in \mathrm{Ann} \Big(\bigoplus_{j=1}^N E^{(i_j)}_{m^{(i_1\dots i_N)}}\oplus F^{(i_N)}_{m^{(i_1\dots i_N)}}\Big ) \Big \}.
\end{split}
\end{align}
Note that the condition $\tilde  B^{(i_N), \mathrm{v}}_{\eta_{m^{(i_1\dots i_N)}}}  =0$ is already implied by the others. Now, $
\dim E^{(i_j)}_{m^{(i_1\dots i_N)}} = \dim G_{x^{(i_{j-1})}} \cdot x ^{(i_j )}$.
Since for  $\sigma_{i_1}\cdots \sigma_{i_N}\not=0$ the $G$-orbit of $m^{(i_1\dots i_N)}$ is of principal type $G/ H_L$ in $M$, one computes in this case
\begin{align*}
\kappa =& \dim G \cdot m^{(i_1\dots i_N)}= \dim T_{m^{(i_1 \dots i_N)}} ( G\cdot m^{(i_1 \dots i_N)})\\
=&\dim [E^{(i_1)}_{m^{(i_1 \dots i_N)}}\oplus  \bigoplus _{j=2}^N \tau_{i_1}\dots \tau_{i_{j-1}} E^{(i_j)} _{m^{(i_1 \dots i_N)}}  \oplus  \tau_{i_1}\dots \tau_{i_N} F^{(i_N)}_{m^{(i_1 \dots i_N)}}] \\
 =& \sum_{j=1} ^N \dim E^{(i_j)}_{m^{(i_1\dots i_N)}} + \dim F^{(i_N)}_{m^{(i_1\dots i_N)}}.
\end{align*}
But since the dimension of the spaces $E^{(i_j)}_{m^{(i_1 \dots i_N)}}$ and $F^{(i_N)}_{m^{(i_1 \dots i_N)}}$  does not depend on the variables $\sigma_{i_j}$, we obtain the equality 
\bq
\label{eq:kappa}
\kappa=\sum_{j=1} ^N \dim E^{(i_j)} _{m^{(i_1 \dots i_N)}}+ \dim F^{(i_N)}_{m^{(i_1 \dots i_N)}}
\eq
for arbitrary ${m^{(i_1 \dots i_N)}}$. Note that, in contrast, the dimension of  $T_{m^{(i_1 \dots i_N)}} ( G\cdot m^{(i_1 \dots i_N)})$ collapses, as soon as one of the $\tau_{i_j}$ becomes zero. Since the annihilator of a subspace of $T_m M$ is itself a linear subspace of $T^\ast_m M$,  we arrive at a vector bundle with $(n-\kappa)$-dimensional fiber that is locally given by the trivialization 
\bqn 
\Big (\sigma_{i_j}, x^{(i_j)},  \tilde v ^{(i_N)}, \mathrm{Ann}  \big(\bigoplus _{j=1}^N  E^{(i_j)} _{m^{(i_1 \dots i_N)}}  \oplus  F^{(i_N)}_{m^{(i_1 \dots i_N)}}\big ) \Big )\mapsto (\sigma_{i_j}, x^{(i_j)},  \tilde v ^{(i_N)}).
\eqn
Consequently, by equation \eqref{eq:C} we see that $\Crit(\, ^{(i_1\dots i_N)} \phw)$ is equal to the fiber product of the mentioned vector bundle with the isotropy algebra bundle given by the local trivialization
\bqn 
(\sigma_{i_j}, x^{(i_j)},  \tilde v ^{(i_N)}, \g_{\tilde v ^{(i_N)}})\mapsto (\sigma_{i_j}, x^{(i_j)},  \tilde v ^{(i_N)}).
\eqn
Lastly, since by equation \eqref{eq:G} we have $\g_{\tilde v ^{(i_N)}}=\g_{m^{(i_1,\dots, i_N)}}$ in case that all $\sigma_{i_j}$ are different from zero, we necessarily have $\dim \g_{\tilde v ^{(i_N)}}=d-\kappa$, which concludes the proof of the theorem.
\end{proof}

\section{Phase analysis of the weak transform. The second fundamental theorem}

In this section, we shall prove the second fundamental theorem in the derivation of equivariant spectral asymptotics for compact group actions.  We begin with  the following general observation. Let $M$ be a $n$-dimensional Riemannian manifold, and $C$ the critical set of a function $\psi \in \Cinft(M)$, which is assumed to be a smooth submanifold in a chart $\mathcal{O} \subset M$. Let further 
\bqn
\alpha:(x,y) \mapsto p, \qquad \beta:(q_1,\dots, q_n) \mapsto m, 	\qquad m \in \mathcal{O},
\eqn
be two systems of  local coordinates on $\mathcal{O}$, such that $\alpha(x,y) \in C$ if and only if $y=0$. One computes
\begin{align*}
\gd_{y_l}(\psi \circ \alpha)(x,y)=\sum_{i=1}^n \frac{\gd(\psi \circ \beta)}{\gd q_i} (\beta^{-1}\circ \alpha(x,y)) \, \gd _{y_l} (\beta^{-1} \circ \alpha)_i (x,y),
\end{align*}
as well as
\begin{align*}
\gd_{y_k} \gd_{y_l} (\psi \circ \alpha)(x,y)&=\sum_{i=1}^n \frac{\gd(\psi \circ \beta)}{\gd q_i} (\beta^{-1}\circ \alpha(x,y)) \, \gd_{y_k}  \gd _{y_l} (\beta^{-1} \circ \alpha)_i (x,y)\\
&+ \sum_{i,j=1}^n \frac{\gd^2(\psi \circ \beta)}{\gd q_i \gd{q_j}} (\beta^{-1}\circ \alpha(x,y))  \gd_{y_k}(\beta^{-1} \circ\alpha)_j(x,y) \,  \gd_{y_l}(\beta^{-1} \circ\alpha)_i(x,y).
\end{align*}
Since
\bqn
\alpha_{\ast,(x,y)}(\gd_{y_k})=\sum_{j=1}^n \gd_{y_k} (\beta^{-1} \circ \alpha)_j(x,y) \, \beta_{\ast,(\beta^{-1} \circ \alpha)(x,y)}(\gd_{q_j}),
\eqn
this implies
\bq
\label{eq:Hess}
\gd_{y_k} \gd_{y_l} (\psi \circ \alpha)(x,0)= \mathrm{Hess}\,  \psi_{|\alpha(x,0)} (\alpha_{\ast,(x,0)}(\gd_{y_k}),\alpha_{\ast,(x,0)}(\gd_{y_l})),
\eq
by definition of the Hessian. Let us now write $x=(x',x'')$, and consider the restriction of $\psi$ onto  the $\Cinft$-submanifold
\bdm
M_{c'}=\mklm { m \in \mathcal{O}: m=\alpha(c',x'',y)}.
\edm
We write $\psi_{c'}=\psi_{|M_{c'}}$, and denote the critical set of $\psi_{c'}$ by $C_{c'}$, which contains $C \cap M_{c'}$ as a subset. Introducing on $M_{c'}$ the local coordinates
\bqn
\alpha':(x'',y) \mapsto \alpha(c',x'',y),
\eqn
we obtain 
\bqn
\gd_{y_k} \gd_{y_l} (\psi_{c'} \circ \alpha')(x'',0)= \mathrm{Hess}\,  \psi_{c'|\alpha(x'',0)} (\alpha'_{\ast,(x'',0)}(\gd_{y_k}),\alpha'_{\ast,(x'',0)}(\gd_{y_l})).
\eqn
Let us now assume $C_{c'}=C \cap M_{c'}$, a transversal intersection.  Then $C_{c'}$ is a submanifold of $M_{c'}$, and the normal space to $C_{c'}$ as a submanifold of $M_{c'}$ at a point $\alpha'(x'',0)$ is spanned by the vector fields $\alpha'_{\ast,(x'',0)} (\gd _{y_k})$.
Since clearly
\bqn
\gd_{y_k} \gd_{y_l} (\psi_{c'} \circ \alpha')(x'',0)=\gd_{y_k} \gd_{y_l} (\psi \circ \alpha)(x,0),\qquad x=(c',x''),
\eqn
we thus have proven the following
\begin{lemma}
\label{lemma:A}
Assume that $C_{c'}=C \cap M_{c'}$. Then the restriction
\bqn
\mathrm{Hess} \, \psi({\alpha(c',x'',0)})_{|N_{\alpha(c',x'',0)}C}
\eqn
of the Hessian of $\psi$ to the normal space $N_{\alpha(c',x'',0)}C$ defines a non-degenerate quadratic form if, and only if the restriction
\bqn
\mathrm{Hess} \, \psi_{c'}({\alpha'(x'',0)})_{|N_{\alpha'(x'',0)}C_{c'}}
\eqn
of the Hessian of $\psi_{c'}$ to the normal space $N_{\alpha'(x'',0)}C_{c'}$ defines a non-degenerate quadratic form.
\end{lemma} \qed

Let us now state the second fundamental theorem, the notation being the same as in the previous sections.

\begin{theorem}
\label{thm:II}
Let
\bqn
\, ^{(i_1\dots i_N)} \tilde \psi^{tot}=\tau_{i_1} \dots \tau_{i_N} \, ^{(i_1\dots i_N)}\tilde \psi^ {wk, \, pre}=\tau_{i_1}(\sigma) \dots \tau_{i_N}(\sigma) \, ^{(i_1\dots i_N)}\tilde \psi^ {wk}
\eqn
denote the factorization of the phase function after $N$ iteration steps along the isotropy branch $ ((H_{i_1}), \dots, (H_{i_{N+1}})=(H_L))$. By construction, for $\tau_{i_j}\not=0$, $1\leq j\leq N$, the $G$-orbit through $m^{(i_1\dots i_N)}$ is of principal type $G/H_L$. Then, for each point of the critical manifold $\Crit(\, ^{(i_1\dots i_N)}\tilde \psi^ {wk})$,  the restriction of 
\bqn 
\mathrm{Hess} \, ^{(i_1\dots i_N)}\tilde \psi^ {wk}
\eqn
to the normal space to $\Crit(\, ^{(i_1\dots i_N)}\tilde \psi^ {wk})$ at the given point defines a non-degenerate symmetric bilinear form.
\end{theorem}

For the proof of Theorem \ref{thm:II} we need  the following
\begin{lemma}
\label{lemma:Reg}
 Let  $(\eta,X) \in \Crit{(\psi)}$, and  $ \pi(\eta) \in M({H_L})$. Then $(\eta,X) \in \mathrm{Reg} \,\Crit{(\psi)}$. Furthermore, the restriction of the Hessian of $\psi$ at the point $(\eta,X)$ to the normal space ${N_{(\eta,X)} \mathrm{Reg} \,\Crit{(\psi)}}$ defines  a non-degenerate quadratic form. 
\end{lemma}
\begin{proof}
The first assertion is clear from \eqref{eq:x} - \eqref{eq:z}.  To see the second, note that by \eqref{eq:y}
\bqn 
\eta \in \Omega \cap T^\ast M(H_{L}), \tilde X_{\pi(\eta)}=0 \quad \Longrightarrow \quad \tilde X _\eta =0.
\eqn
Let now $\mklm{q_1,\dots, q_n}$ be local coordinates on $M$, $m=m(q)$, and write $\eta_m = \sum p_i (dq_i)_m$, $X=\sum s_i X_i$, where $\mklm{X_1,\dots,X_d}$ denotes a basis of $\g$. Then 
\bqn 
\psi(\eta,X)=\sum p_i (d q_i)_m (\tilde X_m),
\eqn
and 
\bqn 
\gd_p \psi (\eta, X) =0 \quad \Longleftrightarrow  \quad \tilde X_m =0, \qquad \quad \gd_s \psi(\eta, X) =0 \quad \Longleftrightarrow \quad \eta \in \Omega.
\eqn 
On  $T^\ast M(H_{L})\times \g$ we therefore get
\bqn 
\gd_{p,s} \psi (\eta, X) =0 \quad \Longrightarrow  \quad \gd_q \psi(\eta, X) =0.
\eqn 
Let $\psi_q(p,s)$ denote the phase function regarded as a function of the coordinates $p,s$ alone, while $q$ is regarded as a parameter. 
Lemma \ref{lemma:A} then implies that  on $T^\ast M(H_{L})\times \g$ the study of the transversal Hessian of $\psi$ can be reduced to the study of the transversal Hessian of $\psi_q$. Now, with respect to the coordinates $s,p$, the Hessian of $\psi_q$ is given by 
\bqn 
\left ( \begin{array}{cc}
0 & (dq_i)_m((\tilde X_{j})_m) \\ (dq_j)_m((\tilde X_{i})_m) & 0\\
\end{array} \right ).
\eqn 
A computation then shows that the kernel of the corresponding linear transformation is isomorphic to $T_{p,s}( \Crit \, \psi_q)=\mklm{(\tilde p,\tilde s): \sum \tilde p_j (dq_j)_{m(q)} \in \mathrm{Ann}(T_{m(q)} ( G \cdot m(q))), \sum \tilde s_j X_j \in \g_{m(q)}}$. The lemma now follows with the following general observation.
Let $\mathcal{B}$ be a symmetric bilinear form on an $n$-dimensional $\mathbb{K}$-vector space $V$, and $B=(B_{ij})_{i,j}$ the corresponding Gramsian matrix with respect to a basis $\mklm{v_1,\dots,v_n}$ of $V$ such that 
\bqn 
\mathcal{B}(u,w) = \sum_{i,j} u_i w_j B_{ij}, \qquad  u=\sum u_i v_i, \quad w=\sum w_i v_i.
\eqn
We denote the linear operator given by $B$ with the same letter, and write 
\bqn 
V =\ker B \oplus W.
\eqn
Consider the restriction $\mathcal{B}_{|W \times W}$ of $\mathcal{B}$ to $W\times W$, and assume that $\mathcal{B}_{|W\times W}(u,w) =0$ for all $u \in W$, but $w\not=0$. Since the Euclidean scalar product in $V$ is non-degenerate, we necessarily must have $Bw=0$, and consequently $ w \in \ker  B \cap W=\mklm{0}$, which is a contradiction. Therefore $\mathcal{B}_{|W \times W}$ defines a non-degenerate symmetric bilinear form. 
\end{proof}

\begin{proof}[Proof of second fundamental theorem] Let us begin by noting that  for $\sigma_{i_1} \cdots \sigma_{i_N}\not=0$,  the sequence of monoidal transformations $\zeta= \zeta_{i_1} \circ \zeta_{i_1i_2} \circ \dots \circ \zeta_{i_1\dots i_N} \circ \delta_{i_1\dots i_N}$  constitutes a diffeomorphism, so that by the previous lemma the restriction of 
\bqn 
\mathrm{Hess} ^{(i_1\dots i_N)} \tilde \psi^{tot} (\sigma_{i_j},x^{(i_j)},\tilde v^{(i_N)}, \alpha^{(i_j)}, \beta^{(i_N)},p)
\eqn
to the normal space of 
\bqn 
\mathrm{Crit}(\, ^{(i_1\dots i_N)} \psi^{tot})_{\sigma_{i_1}\cdots \sigma_{i_N}\not=0}
\eqn
defines a non-degenerate quadratic form.  Next, one computes
\begin{align*}
\left (\frac{\gd^2 \, ^{(i_1\dots i_N)}\tilde \psi^ {tot}}{\gd \gamma_k\gd \gamma_l}  \right )_{k,l} & = \tau_{i_1}(\sigma) \cdots  \tau_{i_N }(\sigma) 
\left (\frac{\gd^2 \, ^{(i_1\dots i_N)}\tilde \psi^ {wk}}{\gd \gamma_k\gd \gamma_l}  \right )_{k,l} \\ &+
\left (\begin{array}{cc}
\left (   \frac{ \gd ^2 (\tau_{i_1}(\sigma) \cdots  \tau_{i_N }(\sigma))}{\gd \sigma_{i_r}\sigma_{i_s}} \right )_{r,s} & 0 \\ 0 & 0
\end{array}\right ) \, ^{(i_1\dots i_N)}\tilde \psi^ {wk} +R
\end{align*}
where $R$ represents a matrix whose entries contain first order derivatives of $^{(i_1\dots i_N)}\tilde \psi^ {wk}$ as factors. But since 
\bqn 
\mathrm{Crit}(\, ^{(i_1\dots i_N)} \psi^{tot})_{\sigma_{i_1}\cdots \sigma_{i_N}\not=0} =\Crit(^{(i_1\dots i_N)}\tilde \psi^ {wk})_{|\sigma_{i_1}\cdots \sigma_{i_N}\not=0}, 
\eqn
 we conclude that  the transversal Hessian of $^{(i_1\dots i_N)}\tilde \psi^ {wk}$ does not degenerate along the manifold $\Crit(^{(i_1\dots i_N)}\tilde \psi^ {wk})_{|\sigma_{i_1}\cdots \sigma_{i_N}\not=0}$. Therefore, it  remains to study the transversal Hessian of $^{(i_1\dots i_N)}\tilde \psi^ {wk}$ in the case that any of the $\sigma_{i_j}$ vanishes. Now, the proof of the first fundamental theorem,  in particular  \eqref{eq:IV},  showed that
\bqn 
\gd _{p, \alpha^{(i_1)}, \dots, \alpha^{(i_N)}, \beta^{(i_N)}} \, ^{(i_1\dots i_N)}\tilde \psi^ {wk}=0 \quad \Longrightarrow \quad \gd _{\sigma_{i_1}, \dots \sigma_{i_N}, x^{(i_1)}, \dots, x^{(i_N)},\tilde v^{(i_N)}} \, ^{(i_1\dots i_N)}\tilde \psi^ {wk}=0.
\eqn
 If therefore 
$$
\, ^{(i_1\dots i_N)}\tilde \psi^ {wk}_{\sigma_{i_j}, x^{(i_j)},\tilde v^{(i_N)}}(\alpha^{(i_j)}, \beta^{(i_N)},p)
$$ 
denotes the weak transform of the phase function $\psi$ regarded as a function of the variables $(\alpha^{(i_1)},\dots, \alpha^{(i_N)}, \beta^{(i_N)},p)$ alone, while the variables $(\sigma_{i_1},\dots,\sigma_{i_N}, x^{(i_1)},\dots, x^{(i_N)},\tilde v^{(i_N)})$ are kept fixed,
\bqn 
\Crit \big ( \, ^{(i_1\dots i_N)}\tilde \psi^ {wk}_{\sigma_{i_j}, x^{(i_j)},\tilde v^{(i_N)}}\big )=\Crit \big ( \, ^{(i_1\dots i_N)}\tilde \psi^ {wk}\big )  \cap \mklm{\sigma_{i_j}, x^{(i_j)},\tilde v^{(i_N)} = \, \, \text{constant}}. 
\eqn
Thus, the critical set of $\, ^{(i_1\dots i_N)}\tilde \psi^ {wk}_{\sigma_{i_j}, x^{(i_j)},\tilde v^{(i_N)}}$ is equal to the fiber over $(\sigma_{i_j}, x^{(i_j)},\tilde v^{(i_N)})$ of the vector bundle
\bqn 
\Big ((\sigma_{i_j}, x^{(i_j)},\tilde v^{(i_N)}), \g _{\tilde v^{(i_N)}} \times \mathrm{Ann} \big ( \bigoplus\limits_{j=1}^N E^{(i_j)}_{m^{(i_1\dots i_N)}} \oplus F^{(i_N)}_{m^{(i_1\dots i_N)}} \big ) \Big ) \mapsto  (\sigma_{i_j}, x^{(i_j)},\tilde v^{(i_N)}),
\eqn
and in particular  a smooth submanifold. Lemma \ref{lemma:A} then implies that the study of the transversal Hessian  of $\, ^{(i_1\dots i_N)}\tilde \psi^ {wk}$ can be reduced to the study of the transversal Hessian of $\, ^{(i_1\dots i_N)}\tilde \psi^ {wk}_{\sigma_{i_j}, x^{(i_j)},\tilde v^{(i_N)}}$. The crucial fact is now contained in the following 
\begin{proposition}
\label{prop:1}
Assume that 
%for $\tau_{i_j}\not=0$, $1\leq j\leq N$, the $G$-orbit through $x^{(i_1\dots i_N)}$ is of principal type $G/H_L$, and let
$\sigma_{i_1} \cdots \sigma_{i_N}=0$. Then 
\bqn 
\ker  \mathrm{Hess} \, ^{(i_1\dots i_N)}\tilde \psi^ {wk}_{\sigma_{i_j}, x^{(i_j)},\tilde v^{(i_N)}}(0,\dots, 0, \beta^{(i_N)},p)\simeq T_{(0, \dots, 0,\beta^{(i_N)},p)}  \mathrm{Crit} \big (\,^{(i_1\dots i_N)}\tilde \psi^ {wk}_{\sigma_{i_j}, x^{(i_j)},\tilde v^{(i_N)}} \big )
\eqn
for all $(0,\dots, 0, \beta^{(i_N)},p) \in \mathrm{Crit} \big (\,^{(i_1\dots i_N)}\tilde \psi^ {wk}_{\sigma_{i_j}, x^{(i_j)},\tilde v^{(i_N)}} \big )$, and arbitrary $x^{(i_j)}$, $\tilde v^{(i_j)}$.
\end{proposition}
\begin{proof}
With \eqref{eq:B} one computes
\begin{align*}
\gd _{p_r} \, ^{(i_1\dots i_N)}\tilde \psi^ {wk}&=  \, dq_r \Big (\widetilde{ A^{(i_1)}}_{x^{(i_1)}} + \sum_{j=2}^N  \lambda(A^{(i_j)})  x ^{(i_j)} + \lambda(B^{(i_N)})  \tilde v ^{(i_N)}  \Big ). 
\end{align*}
The second derivatives therefore read
\begin{align*}
\gd_{p_r} \gd _{p_s} \, ^{(i_1\dots i_N)}\tilde \psi^ {wk}_{\sigma_{i_j}, x^{(i_j)},\tilde v^{(i_N)}} &=0, \\
\gd_{\alpha^{(i_1)}_s} \gd _{p_r} \, ^{(i_1\dots i_N)}\tilde \psi^ {wk}_{\sigma_{i_j}, x^{(i_j)},\tilde v^{(i_N)}} &=dq_r((\tilde A^{(i_1)}_{s})_{x^{(i_1)}}), \\
\gd_{\alpha^{(i_j)}_s} \gd _{p_r} \, ^{(i_1\dots i_N)}\tilde \psi^ {wk}_{\sigma_{i_j}, x^{(i_j)},\tilde v^{(i_N)}} &=dq_r(\lambda(A_s^{(i_j)}) x^{(i_j)}), \\
\gd_{\beta^{(i_N)}_s} \gd _{p_r} \, ^{(i_1\dots i_N)}\tilde \psi^ {wk}_{\sigma_{i_j}, x^{(i_j)},\tilde v^{(i_N)}} &=dq_r(\lambda(B_s^{(i_N)}) \tilde v^{(i_N)}). 
\end{align*}
Next, one has
\begin{align*}
\gd_{\alpha^{(i_j)}_s} \, ^{(i_1\dots i_N)}\tilde \psi^ {wk}&=\sum p_i dq_i (\lambda (A_s^{(i_j)}) x^{(i_j)}), \qquad  j=2,\dots, N,
\end{align*}
and similar expressions for the $\alpha^{(i_1)}$-derivatives, so that for $\sigma_{i_1}\cdots \sigma_{i_j}=0$ all the second order derivatives involving $\alpha^{(i_j)}$  must vanish, except the ones that were already computed. Finally, the computation of the $\beta^{(i_N)}$-derivatives yields
\bqn 
\gd_{\beta^{(i_N)}_r} \gd _{\beta^{(i_N)}_s} \, ^{(i_1\dots i_N)}\tilde \psi^ {wk}_{\sigma_{i_j}, x^{(i_j)},\tilde v^{(i_N)}} =0.
\eqn
Collecting everything we see that for $\sigma_{i_1}\cdots \sigma_{i_j}=0$, the Hessian of the function $\, ^{(i_1\dots i_N)}\tilde \psi^ {wk}_{\sigma_{i_j}, x^{(i_j)},\tilde v^{(i_N)}}$  with respect to the coordinates $\alpha^{(i_j)}, \beta^{(i_j)},p$ is given on its critical set by the matrix
\bqn 
\left ( \begin{array}{ccccc}
0 & dq_r((\tilde A^{(i_1)}_{s})_{x^{(i_1)}}) & \dots & dq_r(\lambda(A_s^{(i_N)}) x^{(i_j)})& dq_r(\lambda(B_s^{(i_N)}) \tilde v^{(i_N)}) \\
\, dq_s((\tilde A^{(i_1)}_{r})_{x^{(i_1)}})& 0 & \dots & 0 & 0 \\
\vdots & \vdots &\vdots &\vdots &\vdots \\
\,dq_s(\lambda(A_r^{(i_N)}) x^{(i_j)})& 0 & \dots & 0 & 0 \\
 \,dq_s(\lambda(B_r^{(i_N)}) \tilde v^{(i_N)})& 0 & \dots & 0 & 0
\end{array} \right ). 
\eqn
Let us now compute the kernel of the linear transformation corresponding to this matrix. Cleary,  the vector $(\tilde p, \tilde \alpha^{(i_1)}, \dots, \tilde \alpha^{(i_N)}, \tilde \beta^{(i_N)})$ lies in the kernel if and only if

\medskip
\begin{tabular}{ll}
{(a)} & $\sum \tilde \alpha_r^{(i_1)} (\tilde A_r^{(i_1)})_{ x^{(i_1)}}+  \dots +\sum \tilde \alpha_r^{(i_N)} \lambda (A_r^{(i_N)}) x^{(i_N)}+ \sum \tilde \beta_r^{(i_N)} \lambda(B_r^{(i_N)}) \tilde v^{(i_N)}=0$ ; \\[2pt]
{(b)} &  $\sum \tilde p_s dq_s((\tilde Y^{(i_1)})_{x^{(i_1)}})=0$ for all $Y^{(i_1)} \in \g_{x^{(i_1)}}^\perp$,   $\sum \tilde p_s dq_s(\lambda ( \g_{x^{(i_j)}}^\perp)x^{(i_j)})=0$, $2 \leq j \leq N$;\\[2pt]
{(c)} &$\sum \tilde p_s dq_s(\lambda ( \g_{x^{(i_N)}})\tilde v^{(i_N)})=0$.
\end{tabular}
\medskip

Let $E^{(i_j)}$,  $F^{(i_N)}$, and  $V^{(i_1\dots i_N)}$ be defined as in \eqref{eq:EF} and \eqref{eq:V}. Then
\bqn 
\sum \tilde \alpha_r^{(i_j)} (\tilde A_r^{(i_1)})_{x^{(i_1)}}+  \dots \sum \tilde \alpha_r^{(i_N)}\lambda( A_r^{(i_N)}) x^{(i_N)}+ \sum \tilde \beta_r^{(i_N)} \lambda( B_r^{(i_N)}) \tilde v^{(i_N)} \in \bigoplus_{j=1}^N E^{(i_j)}_{x^{(i_1)}} \oplus F^{(i_N)}_{x^{(i_1)}},
\eqn
so that for condition (a) to hold, it is  necessary  and sufficient that 
\bqn 
\tilde \alpha^{(i_j)} =0, \quad 1 \leq j \leq N, \qquad \sum \tilde \beta_r^{(i_N)} \lambda(B_r^{(i_N)}) \tilde v^{(i_N)}=0.
\eqn
Since $\g_{x^{(i_j)}}^\perp\subset \g_{x^{(i_{j-1})}}$,  condition (b) is equivalent to $\sum \tilde p_s (dq_s)_{x^{(i_1)}} \in \mathrm{Ann}( E^{(i_j)}_{x^{(i_1)}})$ for al $j=1,\dots, N$. Similarly, condition (c) is equivalent to $\sum \tilde p_s (dq_s)_{x^{(i_1)}} \in \mathrm{Ann}( F^{(i_N)}_{x^{(i_1)}})$.  On the other hand, by \eqref{eq:C},
\begin{gather*}
T_{(0, \dots, 0,\beta^{(i_N)},p)}  \mathrm{Crit} \big (\,^{(i_1\dots i_N)}\tilde \psi^ {wk}_{\sigma_{i_j}, x^{(i_j)},\tilde v^{(i_N)}} \big )= \Big \{( \tilde \alpha^{(i_1)}, \dots, \tilde \alpha^{(i_N)}, \tilde \beta^{(i_N)}, \tilde p): \tilde \alpha^{(i_j)}=0, \\  \sum \tilde  \beta^{(i_N)}_r \lambda(B^{(i_N)}_r) \in \g_{\tilde v^{(i_N)}}, \, \sum \tilde p_s (dq_s)_{x^{(i_1)}} \in \mathrm{Ann}\Big ( \bigoplus_{j=1}^N E^{(i_j)}_{x^{(i_1)}}\oplus F^{(i_N)}\Big ) \Big \},
\end{gather*}
and the proposition follows.
\end{proof}
The previous proposition now  implies that for $\sigma_{i_1}\cdots \sigma_{i_N}=0$
\bqn 
\mathrm{Hess} \,^{(i_1\dots i_N)}\tilde \psi^ {wk}_{\sigma_{i_j}, x^{(i_j)},\tilde v^{(i_N)}}(0, \dots, 0,\beta^{(i_N)},p)_{|N_{(0, \dots, 0,\beta^{(i_N)},p)}  \mathrm{Crit} \big (\,^{(i_1\dots i_N)}\tilde \psi^ {wk}_{\sigma_{i_j}, x^{(i_j)},\tilde v^{(i_N)}} \big )}
\eqn
defines a non-degenerate symmetric bilinear form for all points $(0, \dots, 0,\beta^{(i_N)},p)$ lying in the critical set of $\,^{(i_1\dots i_N)}\tilde \psi^ {wk}_{\sigma_{i_j}, x^{(i_j)},\tilde v^{(i_N)}}$, and the second fundamental theorem follows with  Lemma \ref{lemma:A}.
\end{proof}
We are now in position to give an asymptotic description of the integral $I(\mu)$. But before, it might be in place to say a few words about the desingularization process.

\section{Resolution of singularities and the stationary phase theorem}
\label{sec:6}

Let $M$ be a smooth variety, $\mathcal{O}_M$ the structure sheaf of rings of $M$, and $I \subset \mathcal{O}_M$ an ideal sheaf. The aim in the theory of resolution of singularities is  to construct a birational morphism $\pi: \tilde M \rightarrow M$ such that $\tilde M$ is smooth, and the pulled back ideal sheaf $\pi^\ast I$ is locally principal. This is called the \emph{principalization} of $I$, and implies resolution of singularities. That is, for every quasi-projective variety $X$, there is a smooth variety $\tilde X$, and a birational and projective morphism $\pi:\tilde X \rightarrow X$. Vice versa, resolution of singularities implies principalization.

Consider next the derivative  $D(I)$ of $I$, which is the sheaf ideal that is generated by all derivatives of elements of $I$. Let further $Z \subset M$ be a smooth subvariety, and $\pi: B_Z M \rightarrow M$ the corresponding monoidal transformation with center $Z$ and exceptional divisor $F \subset B_ZM$. Assume that $(I,m)$ is a  marked ideal sheaf with $m \leq \mathrm{ord}_Z I$. The total transform $\pi^\ast I$ vanishes along $F$ with multiplicity $\mathrm{ord}_Z I$, and by removing the ideal sheaf $\mathcal{O}_{B_ZM}(-\mathrm{ord}_Z I \cdot F)$ from $\pi^\ast I$  we obtain the \emph{birational, or weak transform} $\pi_\ast^{-1} I$ of $I$. Take local coordinates $(x_1,\dots, x_n)$ on $M$ such that $Z=(x_1= \dots =x_r=0)$. As a consequence,
\bqn
y_1=\frac {x_1}{x_r}, \dots, y_{r-1}=\frac{x_{r-1}}{x_r}, y_r=x_r, \dots,  y_n=x_n
\eqn
define local coordinates on $B_Z M$, and for $(f,m) \in (I,m)$ one has
\bqn
\pi_\ast^{-1} (f(x_1,\dots,x_n),m)= (y_r^{-m} f (y_1y_r, \dots y_{r-1}y_r, y_r, \dots, y_n),m).
\eqn
By computing the first derivatives of $\pi_\ast^{-1} (f(x_1,\dots,x_n),m)$, one then sees that for any composition  $\Pi:\tilde M \rightarrow M$ of blowing-ups of order greater or equal than $m$, 
\bqn
\Pi^{-1}_\ast ( D(I,m)) \subset D(\Pi^{-1}_\ast(I,m)),
\eqn
see Koll\'{a}r \cite{kollar}, Theorem 71. 

Let us now come back to our situation, and consider on $T^\ast M\times \g$ the ideal $I_\psi=(\psi)$ generated by the phase function $\psi=\mathbb{J}(\eta)(X)$, together with its vanishing set $V_\psi$. The derivative of $I$ is given by $D(I_\psi) = I_{\mathcal{C}}$, 
where $I_{\mathcal{C}}$ denotes the vanishing ideal of the critical set $\mathcal{C}=\Crit(\psi)$, and by the implicit function theorem $\mathrm{Sing} \,V_\psi \subset V_\psi \cap \mathcal{C}=\mathcal{C}$. Let  $((H_{i_1}), \cdots, (H_{i_{N+1}})=(H_L))$ be an arbitrary branch of isotropy types, and consider the corresponding sequence of monoidal transformations $(\zeta_{i_1} \circ \zeta_{i_1i_2} \circ \cdots \circ \zeta_{i_1\dots i_{N}}) \otimes \id_{fiber}$. Compose it with the sequence of monoidal transformations $\delta_{i_1\dots i_N}$, and denote the resulting transformation by $\zeta$. We then have the diagram

\bqn 
 \begin{array}{ccccc}
\zeta^\ast (I_{\mathcal{C}})  & \supset & \zeta^\ast (I_\psi)  &= \prod_{i=1}^N \tau_{i_j}(\sigma) \, \cdot\zeta^{-1}_\ast(I_\psi) & \ni \tau_{i_1}(\sigma) \cdots \tau_{i_N}(\sigma)  \,^{(i_1\dots i_N)}\tilde \psi^ {wk}  \\ [4pt]
\uparrow &  & \uparrow  & &\\[4pt]
I_{\mathcal{C}}  &\supset & I_\psi   & \ni \psi & 
\end{array} 
\eqn

\medskip
\noindent 
According to the previous considerations, we have the inclusion 
\bqn 
\zeta^{-1}_\ast (I_{\mathcal{C}}) \subset D(\zeta^{-1}_\ast (I_\psi)).
\eqn
It is easy to see that $\zeta^{-1}_\ast(I_\psi)$ is not resolved, so that $\prod_{i=1}^N \tau_{i_j}(\sigma) \, \cdot\zeta^{-1}_\ast(I_\psi)$ is only a partial principalization. On the other hand, the first fundamental theorem implies that $D(\zeta^{-1}_\ast (I_\psi))$ is a resolved ideal, 
\bqn
\overline{\mathrm{Crit}(\, ^{(i_1\dots i_N)} \psi^{tot})_{\sigma_{i_1}\dots \sigma_{i_N}\not=0}} =  \mathrm{Crit}(\, ^{(i_1\dots i_N)} \psi^{wk})
\eqn
being a smooth manifold. Nevertheless,  this again results only in a partial resolution $\tilde{\mathcal{C}}$  of $\mathcal{C}$, since the induced global birational transform $ \tilde{\mathcal{C}} \rightarrow \mathcal{C}$ is in general not surjective. This is because of the transformation $\delta_{i_1\dots i_N}$, and the fact that the centers of our monoidal transformations were only chosen in $M \times \g$, to keep the phase analysis of the weak transform of $\psi$ as simple as possible. In turn, the singularities of $\mathcal{C}$ along the fibers of $T^\ast M$ were not completely resolved.

As we shall see in the next section, the principalization of the ideal $I_\psi$ 
\bqn 
\zeta^\ast (I_\psi)  = \tau_{i_1}\cdots \tau_{i_N} \zeta^{-1}_\ast(I_\psi), 
\eqn
and the fact that the weak transform $ \,^{(i_1\dots i_N)}\tilde \psi^ {wk}$ has a clean critical set, are essential for an application of the stationary phase principle in the context of singular equivariant asymptotics, which  is we why had to  consider resolutions of both $\mathcal{C}$ and $V_\psi$ in $T^\ast M \times \g$. By Hironaka's theorem on resolution of singularities,  such resolutions always exist, and are equivalent to the principalization of the corresponding ideals.  But in general, they would not be explicit enough \footnote{In particular, the so-called numerical data of $\zeta$ are not known a priori, which in our case are given in terms of the dimensions $c^{(i_j)}$ and $d^{(i_j)}$.}  to allow an application of the stationary phase theorem. This is the reason why we were forced to construct an explicit, though partial, resolution $\zeta$ of $\mathcal{C}$ and $V_\psi$ in $T^\ast M\times \g $, using as centers  isotropy algebra bundles over sets of maximal singular orbits. Partial desingularizations  of the zero level set $\Omega$ of the moment map and the symplectic quotient $\Omega/G$ have been obtained  e.g. by Meinrenken-Sjamaar \cite{meinrenken-sjamaar} for compact symplectic manifolds with a Hamiltonian compact Lie group action by performing blowing-ups along minimal symplectic suborbifolds containing the strata of maximal depth  in $\Omega$.

\section{Asymptotics for the integrals  $I_{i_1\dots i_N}(\mu)$ }

In this section, we will give an asymptotic description of the integrals $I_{i_1\dots i_N}(\mu)$  defined in \eqref{eq:N}.  Since the considered integrals are absolutely convergent integral, we can interchange the order of integration by Fubini, and write
\bqn
I_{i_1\dots i_N}(\mu)= \int_{(-1,1)^N}  J_{\tau_{i_1}, \dots, \tau_{i_N}}\Big ( \frac \mu{\tau_{i_1}\cdots \tau_{i_N}} \Big ) \prod_{j=1}^N |\tau_{i_j}|^{c^{(i_j)} + \sum _{r=1}^j d^{(i_r)} -1} \d \tau_{i_N} \dots \d \tau_{i_1},
\eqn
where we set
\begin{gather*}
 J_{\tau_{i_1}, \dots, \tau_{i_N}}(\nu)\\= \int  e^{i \, ^{(i_1\dots i_N)} \tilde \psi ^{wk,pre}/\nu}  \, a_{i_1\dots i_N} \,    \Phi_{i_1\dots i_N}   \d(T^\ast _{m^{(i_1\dots i_N)}}W_{i_j})(\eta)  \bigwedge_j dA^{(i_j)}  \d B^{(i_N)}  \d \tilde v^{(i_N)}  \bigwedge_l \d x^{(i_{l})},
\end{gather*}
and introduced the new parameter
\bqn
\nu =\frac \mu {\tau_{i_1}\cdots \tau_{i_N}}.
\eqn

Now,  for an arbitrary $0<\epsilon < T$ to be chosen later we define
\begin{align*}
I_{i_1\dots i_N}^1(\mu)&= \int_{((-1,1)\setminus (-\epsilon,\epsilon))^N}  J_{\tau_{i_1}, \dots, \tau_{i_N}}\Big ( \frac \mu{\tau_{i_1}\cdots \tau_{i_N}} \Big ) \prod_{j=1}^N |\tau_{i_j}|^{c^{(i_j)} + \sum _{r=1}^j d^{(i_r)} -1} \d \tau_{i_N} \dots \d \tau_{i_1},\\
I_{i_1\dots i_N}^2(\mu)&= \int_{(-\epsilon,\epsilon)^N}  J_{\tau_{i_1}, \dots, \tau_{i_N}}\Big ( \frac \mu{\tau_{i_1}\cdots \tau_{i_N}} \Big ) \prod_{j=1}^N |\tau_{i_j}|^{c^{(i_j)} + \sum _{r=1}^j d^{(i_r)} -1} \d \tau_{i_N} \dots \d \tau_{i_1} . 
\end{align*}
\begin{lemma}
\label{lemma:kappa}
One has $c^{(i_j)} + \sum _{r=1}^j d^{(i_r)} -1\geq \kappa$ for arbitrary $j=1,\dots, N$.
\end{lemma}
\begin{proof}
We first note that
\bqn
c^{(i_j)} = \dim (\nu_{i_1\dots i_j})_{x^{(i_j)}} \geq \dim G_{x^{(i_j)}} \cdot m^{(i_{j+1}\dots i_N)} +1.
\eqn
Indeed, $(\nu_{i_1\dots i_j})_{x^{(i_j)}}$ is an orthogonal $G_{x^{(i_j)}}$-space, so that the dimension of the $G_{x^{(i_j)}}$-orbit of $m^{(i_{j+1}\dots i_N)}\in \gamma^{(i_j)}((S^+_{i_1 \dots i_j})_{x^{(i_j)}})$ can be at most $c^{(i_j)}-1$. Now, under the assumption $\sigma_{i_1}\cdots \sigma_{i_N}\not=0$ one has 
\begin{align*}
T_{m^{(i_{j+1}\dots i_N)}} &(G_{x^{(i_j)}} \cdot m^{(i_{j+1}\dots i_N)}) \simeq T_{m^{(i_{1}\dots i_N)}} (G_{x^{(i_j)}} \cdot m^{(i_{1}\dots i_N)}) \\
&= E^{(i_{j+1})}_{m^{(i_1 \dots i_N)}}\oplus  \bigoplus _{k=j+2}^N \tau_{i_{j+1}}\dots \tau_{i_{k-1}} E^{(i_k)} _{m^{(i_1 \dots i_N)}}  \oplus  \tau_{i_{j+1}}\dots \tau_{i_N} F^{(i_N)}_{m^{(i_1 \dots i_N)}}, 
\end{align*}
where the distributions $E^{(i_{j})}$, $F^{(i_{N})}$ where defined in \eqref{eq:EF}. On then computes
\begin{align*}
\dim G_{x^{(i_j)}}  \cdot m^{(i_{j+1}\dots i_N)}= &\dim T_{m^{(i_{j+1}\dots i_N)}} (G_{x^{(i_j)}} \cdot m^{(i_{j+1}\dots i_N)}) \\
=&\sum_{l=j+1} ^N \dim E^{(i_l)} _{m^{(i_1 \dots i_N)}}+ \dim F^{(i_N)}_{m^{(i_1 \dots i_N)}},
\end{align*}
which implies 
\bqn
c^{(i_j)} \geq \sum_{l=j+1} ^N \dim E^{(i_l)}_{m^{(i_1 \dots i_N)}} + \dim F^{(i_N)}_{m^{(i_1 \dots i_N)}} +1.
\eqn
Here we used the same arguments as in the proof of equation \eqref{eq:kappa}. 
On the other hand, one has
\begin{align*}
d^{(i_j)}&=\dim \g_{x^{(i_j)}}^\perp =\dim [\lambda( \g_{x^{(i_j)}}^\perp) \cdot x^{(i_j)} ]=\dim [ \lambda( \g_{x^{(i_j)}}^\perp)  \cdot m^{(i_j\dots i_N)} ]= \dim E^{(i_j)}_{m^{(i_1 \dots i_N)}}.
\end{align*}
The assertion now follows with \eqref{eq:kappa}.
\end{proof}
As a consequence of the lemma, we obtain for $I_{i_1\dots i_N}^2(\mu)$ the estimate
\begin{align}
\label{eq:I2}
\begin{split}
I_{i_1\dots i_N}^2(\mu)&\leq C \int_{(-\epsilon,\epsilon)^N}  \prod_{j=1}^N |\tau_{i_j}|^{c^{(i_j)} + \sum _{r=1}^j d^{(i_r)} -1} \d \tau_{i_N} \dots \d \tau_{i_1} \\
&\leq C \int_{(-\epsilon,\epsilon)^N}  \prod_{j=1}^N |\tau_{i_j}|^{\kappa} \d \tau_{i_N} \dots \d \tau_{i_1} =\frac{2C}{\kappa+1} \epsilon^{N (\kappa+1)}
\end{split}
\end{align}
for some $C>0$. Let us now turn to the integral $I_{i_1\dots i_N}^1(\mu)$. After performing the change of variables $\delta_{i_1\dots i_N}$
one obtains
\begin{align*}
I_{i_1\dots i_N}^1(\mu)&  = \int\limits _{\epsilon < |\tau_{i_j} (\sigma)|< 1} \hspace{-.5cm} J_{\sigma_{i_1}, \dots, \sigma_{i_N}}\Big ( \frac \mu{\tau_{i_1}(\sigma)\cdots \tau_{i_N}(\sigma)} \Big )  \prod_{j=1}^N |\tau_{i_j}(\sigma)|^{c^{(i_j)} + \sum _{r=1}^j d^{(i_r)} -1} \, |\det D\delta_{i_1\dots i_N}(\sigma) | \d \sigma,
\end{align*}
where 
\bqn
 J_{\sigma_{i_1}, \dots, \sigma_{i_N}}(\nu)= \int  e^{i \, ^{(i_1\dots i_N)} \tilde \psi ^{wk}_\sigma /\nu}  \, a_{i_1\dots i_N} \,    \Phi_{i_1\dots i_N}  \d (T^\ast _{m^{(i_1\dots i_N)}} W_{i_j})(\eta)  \bigwedge_j dA^{(i_j)}   \d B^{(i_N)}  \d \tilde v^{(i_N)}  \bigwedge_l  \d x^{(i_{l})} .
\eqn
Here we denoted by $ ^{(i_1\dots i_N)} \tilde \psi ^{wk}_\sigma $ the weak transform of the phase function $\psi$ as a function of the variables $ x^{(i_j)},  \tilde v^{(i_N)}, \alpha^{(i_j)}, \beta^{(i_N)}, p$ alone, while the variables $\sigma=(\sigma_{i_1},\dots \sigma_{i_N})$ are regarded as parameters. The idea is now to make use of the principle of the stationary phase to give an asymptotic expansion of $J_{\sigma_{i_1}, \dots, \sigma_{i_N}}(\nu)$. 

\begin{theorem}[Generalized stationary phase theorem for manifolds]
\label{thm:SP}
Let $M$ be a  $n$-dimensional Riemannian manifold,  $\psi \in \Cinft(M)$ be a real valued phase function,  $\mu >0$, and set
\bqn
I({\mu})=\int_M e^{i\psi(m)/\mu} a(m) \, dm,
\eqn
where $a(m)dm$ denotes  a compactly supported $\Cinft$-density on $M$. Let
$$\mathcal C=\mklm{m \in M: \psi_\ast:T_mM \rightarrow T_{\psi(m)}\R \text{ is zero}}$$
 be the  critical set of the phase function $\psi$, and assume that
\begin{enumerate}
\item $\mathcal{C}$ is a smooth submanifold of $M$  of dimension $p$ in a neighborhood of the support of $a$;
\item for all $m \in \mathcal{C}$, the restriction $\psi''(m)_{|N_m\mathcal{C}}$ of the Hessian of $\psi$ at the point $m$  to the normal space $N_m\mathcal{C}$ is a non-degenerate quadratic form.
\end{enumerate}
\noindent
Then, for all $N \in \N$, there exists a constant $C_{N,\psi}>0$ such that
\bqn
|I(\mu) - e^{i\psi_0/\mu}(2\pi \mu)^{\frac {n-p}{2}}\sum_{j=0} ^{N-1} \mu^j Q_j (\psi;a)| \leq C_{N,\psi} \mu^N \vol (\supp a \cap \mathcal{C}) \sup _{l\leq 2N} \norm{D^l a }_{\infty,M},
\eqn
where $D^l$ is a differential operator on $M$ of order $l$, and $\psi_0$ is the constant value of $\psi$ on $\mathcal{C}$. Furthermore, for each $j$ there exists a constant $\tilde C_{j,\psi}>0$ such that 
\bqn
|Q_j(\psi;a)|\leq \tilde C_{j,\psi}  \vol (\supp a \cap \mathcal{C}) \sup _{l\leq 2j} \norm{D^l a }_{\infty,\mathcal{C}},
\eqn
and, in particular,
\bqn
Q_0(\psi;a)= \int _{\mathcal{C}} \frac {a(m)}{|\det \psi''(m)_{|N_m\mathcal{C}}|^{1/2}} d\sigma_{\mathcal{C}}(m) e^{ i \pi\sigma_{\psi''}},
\eqn
where $\sigma_{\psi''}$ is the constant value of the signature of $\psi''(m)_{|N_m\mathcal{C}}$ for $m$ in $\mathcal{C}$.
\end{theorem}
\begin{proof}
See for instance H\"{o}rmander, \cite{hoermanderI}, Theorem 7.7.5, together with Combescure-Ralston-Robert \cite{combescure-ralston-robert}, Theorem 3.3, as well as Varadarajan \cite{varadarajan97}, pp. 199.
\end{proof}
\begin{remark}
\label{rmk:A}
An examination of the proof of the foregoing theorem shows that the constants $C_{N,\psi}$ are essentially bounded from above by 
\bqn 
\sup_{m \in \mathcal{C} \cap \supp a} \norm {\Big ( \psi''(m)_{|N_m\mathcal{C}}\Big ) ^{-1}}.
\eqn
Indeed, let $\alpha:(x,y) \rightarrow m \in \mathcal{O}\subset M$ be local normal coordinates such that $\alpha(x,y) \in \mathcal{C}$ if, and only if, $y=0$. By \eqref{eq:Hess}, the transversal Hessian $\mathrm{Hess} \, \psi(m)_{|N_m\mathcal{C}}$ is given in these coordinates by the matrix
\bqn 
\Big (\gd _{y_k} \gd _{y_l} (\psi \circ \alpha)(x,0) \Big )_{k,l} \
\eqn
where $m =\alpha (x,0)$. If now the transversal Hessian of $\psi$ is non-degenerate at the point $m=\alpha(x,0)$, then $y=0$ is a non-degenerate critical point of the function  $y \mapsto (\psi \circ \alpha)(x,y)$, and therefore an isolated critical point by the lemma of Morse. As a consequence,  
\bq
\label{eq:est}
\frac{|y|}{|\gd_y (\psi \circ \alpha)(x,y)|} \leq 2 \norm {\Big (\gd_{y_k} \gd _{y_l} (\psi\circ \alpha)(x,0)\Big ) _{k,l}^{-1}}
\eq
for $y$ close to zero. The assertion now follows by applying H\"{o}rmander \cite{hoermanderI}, Theorem 7.7.5,  to the integral
\bqn 
\int_{\alpha^{-1}(\mathcal{O})} e^{i (\psi \circ \alpha) (x,y)/\mu} (a \circ \alpha)(x,y) \d y \d x
\eqn
in the variable $y$, and with $x$ as a parameter, since in our situation the constant $C$ occuring in H\"{o}rmander \cite{hoermanderI}, equation (7.7.12), is precisely bounded by \eqref{eq:est}, if we assume as we may that $a$ is supported near $\mathcal{C}$. A similar observation holds with respect to the constants $\tilde C_{j,\psi}$.
\end{remark}
We arrive now at the following
\begin{theorem}
\label{thm:J}
Let $\sigma=(\sigma_{i_1},\dots, \sigma_{i_N})$ be a fixed set of parameters. Then, for every $\tilde N \in \N$ there exists a constant $C_{\tilde N, ^{(i_1\dots i_N)} \tilde \psi ^{wk}_\sigma}>0$ such that 
\bqn
|J_{\sigma_{i_1},\dots ,\sigma_{i_N}} (\nu) -(2\pi |\nu|)^\kappa\sum_{j=0} ^{\tilde N-1} |\nu|^j Q_j (^{(i_1\dots i_N)} \tilde \psi ^{wk}_\sigma;a_{i_1\dots i_N} \Phi_{i_1\dots i_N})| \leq C_{\tilde N,^{(i_1\dots i_N)} \tilde \psi ^{wk}_\sigma} |\nu|^{\tilde N},
\eqn
with estimates for the coefficients $Q_j$, and an explicit expression for $Q_0$. Moreover,  the constants $C_{\tilde N, ^{(i_1\dots i_N)} \tilde \psi ^{wk}_\sigma}$ and the coefficients $Q_j$ have uniform bounds in $\sigma$.
\end{theorem}
\begin{proof}
As a consequence of the fundamental theorems, and Lemma \ref{lemma:A}, together with the observations preceding Proposition \ref{prop:1}, the phase function $^{(i_1\dots i_N)} \tilde \psi ^{wk}_\sigma$ has a clean critical set, meaning that 
\begin{itemize}
\item the critical set $\Crit( ^{(i_1\dots i_N)} \tilde \psi ^{wk}_\sigma )$ is a $\Cinft$-submanifold of codimension $2\kappa$ for arbitrary $\sigma$;
\item  the transversal Hessian
\bqn
\mathrm{Hess}  \,^{(i_1\dots i_N)} \tilde \psi ^{wk}_\sigma ( x^{(i_j)},  \tilde v^{(i_N)}, \alpha^{(i_j)}, \beta^{(i_N)},p)_{|N_{  (x^{(i_j)},  \tilde v^{(i_N)}, \alpha^{(i_j)}, \beta^{(i_N)},p)}\Crit \big ( ^{(i_1\dots i_N)} \tilde \psi ^{wk}_\sigma \big)}
\eqn 
defines a non-degenerate symmetric bilinear form for arbitrary $\sigma$ at every point  of the critical set of $^{(i_1\dots i_N)} \tilde \psi ^{wk}_\sigma$.
\end{itemize}
Thus, the necessary conditions for applying the principle of the stationary phase to the integral $J_{\sigma_{i_1},\dots,\sigma_{i_N}}(\nu)$ are fulfilled, and we obtain the desired asymptotic expansion by  Theorem \ref{thm:SP}. To see  the existence of the uniform bounds, note that by  Remark \ref{rmk:A} we have
\bqn 
C_{\tilde N, ^{(i_1\dots i_N)} \tilde \psi ^{wk}_\sigma} \leq C'_{\tilde N} \sup_{ x^{(i_j)},  \tilde v^{(i_N)}, \alpha^{(i_j)}, \beta^{(i_N)},p} \norm{\Big ({\mathrm{Hess} \,  ^{(i_1\dots i_N)} \tilde \psi ^{wk}_\sigma}_{|N \Crit ( ^{(i_1\dots i_N)} \tilde \psi ^{wk}_\sigma)}\Big ) ^{-1}}.
\eqn
But since by Lemma \ref{lemma:A} the transversal Hessian 
\bqn 
{\mathrm{Hess} \,  ^{(i_1\dots i_N)} \tilde \psi ^{wk}_\sigma}_{|N_{ (x^{(i_j)},  \tilde v^{(i_N)}, \alpha^{(i_j)}, \beta^{(i_N)},p ) } \Crit ( ^{(i_1\dots i_N)} \tilde \psi ^{wk}_\sigma)}
\eqn 
is given by 
\bqn 
{\mathrm{Hess} \,  ^{(i_1\dots i_N)} \tilde \psi ^{wk}}_{|N_{ (\sigma_{i_j}, x^{(i_j)},  \tilde v^{(i_N)}, \alpha^{(i_j)}, \beta^{(i_N)},p ) } \Crit ( ^{(i_1\dots i_N)} \tilde \psi ^{wk})},
\eqn 
we finally obtain the estimate
\bqn 
C_{\tilde N, ^{(i_1\dots i_N)} \tilde \psi ^{wk}_\sigma} \leq C'_{\tilde N} \sup_{\sigma_{i_j},  x^{(i_j)},  \tilde v^{(i_N)}, \alpha^{(i_j)}, \beta^{(i_N)},p} \norm{\Big ({\mathrm{Hess} \,  ^{(i_1\dots i_N)} \tilde \psi ^{wk}}_{|N \Crit ( ^{(i_1\dots i_N)} \tilde \psi ^{wk})}\Big ) ^{-1}} \leq C_{\tilde N, {i_1\dots i_N}}
\eqn
by a constant independent of $\sigma$. Similarly, one can show the existence of bounds of the form
\bqn 
|Q_j(^{(i_1\dots i_N)} \tilde \psi ^{wk}_\sigma; a_{i_1\dots i_N} \Phi_{i_1\dots i_N}) |\leq \tilde C_{j, {i_1\dots i_N}},
\eqn
with constants $\tilde C_{j, {i_1\dots i_N}}$ independent of $\sigma$. 
\end{proof}

\begin{remark}
Before going on, let us remark that for the computation of the integrals $I_{i_1\dots i_N}^1(\mu)$ it is only necessary to have an asymptotic expansion for the integrals $J_{\sigma_{i_1},\dots,\sigma_{i_N}} (\nu)$ in the case that $\sigma_{i_1} \cdots \sigma_{i_N}\not=0$, which can also be  obtained without the fundamental theorems using only the factorization of the phase function $\psi$ given by the resolution process, together with Lemma \ref{lemma:Reg}. Nevertheless, the main consequence to be drawn from the fundamental theorems is that the constants $C_{\tilde N, ^{(i_1\dots i_N)} \tilde \psi ^{wk}_\sigma}$ and the coefficients $Q_j$ in Theorem \ref{thm:J} have uniform bounds in $\sigma$. 
\end{remark}

As a consequence of Theorem \ref{thm:J}, we obtain for arbitrary $\tilde N\in \N$
\begin{align*}
|J_{\sigma_{i_1},\dots, \sigma_{i_N}} (\nu) & -(2\pi |\nu|)^\kappa Q_0( ^{(i_1\dots i_N)} \tilde \psi ^{wk}_\sigma;a_{i_1\dots i_N} \Phi_{i_1\dots i_N})| \\&\leq 
\Big |J_{\sigma_{i_1},\dots ,\sigma_{i_N}} (\nu) -(2\pi |\nu|)^\kappa\sum_{l=0} ^{\tilde N-1} |\nu|^l Q_l (^{(i_1\dots i_N)} \tilde \psi ^{wk}_\sigma;a_{i_1\dots i_N} \Phi_{i_1\dots i_N})\Big | \\ & + (2\pi |\nu|)^\kappa \sum_{l=1} ^{\tilde N-1} |\nu|^l |Q_l (^{(i_1\dots i_N)} \tilde \psi ^{wk}_\sigma;a_{i_1\dots i_N} \Phi_{i_1\dots i_N})| \leq c_1 |\nu|^{\tilde N}+c_2 |\nu|^\kappa \sum_{l=1}^{\tilde N-1} |\nu|^l
\end{align*}
with constants $c_i>0$ independent  of both $\sigma$ and $\nu$.
From this we deduce
\begin{align*}
\Big |I_{i_1\dots i_N}^1&(\mu)  - (2\pi \mu)^{\kappa}  \int _{\epsilon < |\tau_{i_j}(\sigma) |< 1}   Q_0 \prod_{j=1}^N |\tau_{i_j}(\sigma)|^{c^{(i_j)} + \sum _{r=1}^j d^{(i_r)} -1-\kappa}  |\det D\delta_{i_1\dots i_N}(\sigma) | \d \sigma \Big| \\&\leq c_3 \mu^{\tilde N}   \int_{\epsilon < |\tau_{i_j}(\sigma) |< 1}   \prod_{j=1}^N |\tau_{i_j}(\sigma)|^{c^{(i_j)} + \sum _{r=1}^j d^{(i_r)} -1-\tilde N} \, |\det D\delta_{i_1\dots i_N}(\sigma) | \d \sigma\\
&+ c_4 \mu^{\kappa} \sum_{l=1}^{\tilde N-1} \mu ^l  \int_{\epsilon < |\tau_{i_j}(\sigma) |< 1}   \prod_{j=1}^N |\tau_{i_j}(\sigma)|^{c^{(i_j)} + \sum _{r=1}^j d^{(i_r)} -1-\kappa -l} \, |\det D\delta_{i_1\dots i_N}(\sigma) | \d \sigma\\
& \leq c_5 \mu^{\tilde N} \max \Big \{1, \prod _{j=1}^N  \epsilon^{c^{(i_j)} + \sum _{r=1}^j d^{(i_r)} -\tilde N} \Big \} +c_6 \sum_{l=1}^{\tilde N -1} \mu^{\kappa +l} \max \Big \{1, \prod _{j=1}^N  \epsilon^{c^{(i_j)} + \sum _{r=1}^j d^{(i_r)} -\kappa -l} \Big \}.
\end{align*}
We now set
\bqn 
\epsilon=\mu^{1/N}.
\eqn
Taking into account Lemma \ref{lemma:kappa}, one infers  that the right hand side of the last inequality can be estimated by 
\bqn 
c_5 \max \mklm{ \mu^{\tilde N}, \mu^{\kappa +1}} + c_6 \sum _{l=1}^{\tilde N-1} \max \mklm{ \mu^{\kappa +l}, \mu^{\kappa +1}},
\eqn
so that for sufficiently large $\tilde N \in \N$ we finally obtain an asymptotic expansion for $I_{i_1\dots i_N}(\mu)$ by taking into account \eqref{eq:I2}, and the fact that 
\bqn 
(2\pi \mu)^\kappa \int_{0 < |\tau_{i_j} |< \mu^{1/N}}  Q_0  \prod_{j=1}^N |\tau_{i_j}|^{c^{(i_j)} + \sum _{r=1}^j d^{(i_r)} -1} \d \tau_{i_N} \dots \d \tau_{i_1} = O(\mu^{\kappa+1}). 
\eqn
\begin{theorem}
Let the assumptions of the first fundamental theorem be fulfilled. Then 
\bqn 
I_{i_1\dots i_N}(\mu)=(2 \pi \mu)^\kappa L_{i_1\dots i_N}+ O(\mu^{\kappa+1}),
\eqn
where the leading coefficient $L_{i_1\dots i_N}$ is given by 
\bq
\label{eq:L}
L_{i_1\dots i_N}=\int_{\Crit( ^{(i_1\dots i_N)} \tilde \psi^{wk})} \frac { a_{i_1\dots i_N} \Phi_{i_1\dots i_N} \, d\Crit( ^{(i_1\dots i_N)} \tilde \psi^{wk})} {|\mathrm{Hess} ( ^{(i_1\dots i_N)} \tilde \psi^{wk})_{N\Crit( ^{(i_1\dots i_N)} \tilde \psi^{wk})}|^{1/2}},
\eq
where $ d\Crit( ^{(i_1\dots i_N)} \tilde \psi^{wk})$ denotes the induced Riemannian measure.
\end{theorem}\qed

\section{Statement of the main result}
\label{sec:8}

Let us now return to our departing point, that is, the asymptotic behavior of the integral $I(\mu)$ introduced in \eqref{int}. For this, we still have to examine the contributions to $I(\mu)$ coming from integrals of the form
\begin{gather*}
\begin{split}
\tilde I_{i_1\dots i_\Theta}(\mu) = \qquad \qquad \qquad \qquad\qquad \qquad\qquad \qquad\qquad \qquad \\ \int_{M_{i_1}(H_{i_1})\times (-1,1)} \Big [ \int_{\gamma^{(i_1)}((S^+_{i_1})_{x^{(i_1)}})_{i_2}(H_{i_2})\times (-1,1)} \dots \Big [  \int_{\gamma^{(i_{\Theta-1})}((S^+_{i_1\dots i_{\Theta-1}})_{x^{(i_{\Theta-1})}})_{i_{\Theta}}(H_{i_{\Theta}})\times (-1,1)} \\
\Big [ \int_{\gamma^{(i_{\Theta})}((S_{i_1\dots i_{\Theta}}^+)_{x^{(i_\Theta)}})\times \g_{x^{(i_{\Theta})}}\times \g_{x^{(i_{\Theta})}}^\perp \times \cdots \times \g_{x^{(i_{1})}}^\perp\times T^\ast _{m^{(i_1\dots i_\Theta)}} W_{i_1}} e^{i\frac {\tau_1 \dots \tau_\Theta}\mu \, ^{(i_1\dots i_\Theta)} \tilde \psi ^{wk}}  \, a_{i_1\dots i_\Theta} \,   \tilde \Phi_{i_1\dots i_\Theta} \\
 \d (T^\ast _{m^{(i_1\dots i_\Theta)}} W_{i_1})(\eta)  \d A^{(i_1)} \dots  \d A^{(i_\Theta)}  \d B^{(i_\Theta)} \d \tilde v^{(i_\Theta)} \Big ]  \d \tau_{i_\Theta} \d x^{(i_{\Theta})} \dots  \Big ] \d \tau_{i_2} \d x^{(i_{2})} \Big ]\d \tau_{i_1} \d x^{(i_{1})},
\end{split}
\end{gather*}
where $((H_{i_1}), \dots, (H_{i_\Theta}))$ is an arbitrary isotropy branch, and 
\bqn 
a_{i_1\dots i_\Theta}=[ a \, \chi_{i_1}  \circ (\id_{fiber} \otimes \zeta_{i_1} \circ \zeta_{i_1i_2} \circ \dots \circ \zeta_{i_1\dots i_\Theta})] \, [ \chi_{i_1i_2} \circ \zeta_{i_1i_2} \circ \dots\circ \zeta_{i_1\dots i_\Theta}  ] \dots [\chi_{i_1\dots i_\Theta} \circ \zeta_{i_1\dots i_\Theta}]
\eqn
is supposed to have compact support in one of the $\alpha^{(i_\Theta)}$-charts, and
\begin{align*}
\tilde \Phi_{i_1\dots i_\Theta} &=\prod_{j=1}^\Theta |\tau_{i_j}|^{c^{(i_j)}+\sum_r^j d^{(i_r)}-1}\Phi_{i_1\dots i_\Theta},
\end{align*}
 $\Phi_{i_1\dots i_\Theta}$ being a smooth function which does not depend on the variables $\tau_{i_j}$. Now, a computation of the $\xi$-derivatives of $\, ^{(i_1\dots i_\Theta)} \tilde \psi ^{wk}$ in any of the $\alpha^{(i_\Theta)}$-charts shows that $\, ^{(i_1\dots i_\Theta)} \tilde \psi ^{wk}$ has no critical points  there. 
 By the non-stationary phase theorem, see H\"{o}rmander \cite{hoermanderI}, Theorem 7.7.1, one then computes for arbitrary $\tilde N \in \N$
 \begin{align*}
| \tilde I_{i_1 \dots i_\Theta}(\mu)| \leq c_7 \mu^{\tilde N} \int_{\epsilon<|\tau_{i_j}| < 1} \prod_{j=1} ^\Theta |\tau_{i_j}|^{c^{(i_j)}+\sum_r^j d^{(i_r)}-1-\tilde N}d\tau + c_8 \epsilon^{\Theta (\kappa+1)} \leq c_9 \max\mklm{\mu^{\tilde N},\mu^{\kappa+1}},
 \end{align*}
 where we took $\epsilon=\mu^{1/\Theta}$. Choosing $\tilde N$ large enough, we conclude that
 \bqn
 | \tilde I_{i_1 \dots i_\Theta}(\mu)| =O(\mu^{\kappa+1}).
 \eqn
  As a consequence of this we see that, up to terms of order $O(\mu^{\kappa+1})$,  $I(\mu)$ can be written as a sum
\begin{align*}
I(\mu)&=\sum_{k<L} I_k(\mu)+I_L(\mu) =\sum_{k<l<L} I_{kl}(\mu) + \sum_{k<L} I_{kL}(\mu)  +I_L(\mu)\\
&=\sum _N\sum_{i_1<\dots< i_N<  i_{N+1} =L} I_{i_1\dots i_N}( \mu)+\sum_\Theta \sum_{i_1<\dots<i_\Theta<i_{\Theta+1} \not=L} I_{i_1\dots i_\Theta L}( \mu),
\end{align*}
where the first term in the last line is a sum to be taken over all the indices $i_1,\dots, i_N$ corresponding to all possible isotropy branches of the form $(H_{i_1},\dots,  (H_{i_{N}}), (H_{i_{N+1}})=(H_L))$ of varying length $N$, while the second term is a sum over all indices $i_1, \dots, i_\Theta$ corresponding to branches  $(H_{i_1},\dots, (H_{i_{\Theta}}), (H_{i_{\Theta+1}})\not=(H_L))$ of arbitrary length $\Theta$. The asymptotic behavior of the integrals $I_{i_1\dots i_N}(\mu)$ has been determined in the previous section, and it is not difficult to see that the integrals $I_{i_1\dots i_\Theta L}$ have analogous asymptotic descriptions. We are now ready to state and prove the main result of this paper.
\begin{theorem}
Let $M$ be a connected, closed Riemannian manifold, and $G$  a compact, connected Lie group $G$ with Lie algebra $\g$ acting isometrically and effectively on $M$. Consider the oscillatory integral
\bqn
I(\mu)=   \int _{T^\ast M}  \int_{\g} e^{i \psi( \eta,X)/\mu }   a(\eta,X)  \d X \d\eta ,   \qquad \mu >0,
\eqn 
where the phase function 
\bqn 
\psi(\eta,X)=\mathbb{J}(\eta)(X)
\eqn
is given by the moment map $\mathbb{J}:T^\ast M \rightarrow \g^\ast$ corresponding to  the Hamiltonian action on $T^\ast M$, and $d\eta$ is a density on $T^\ast M$, while $dX$ is, up to a constant factor, the Lebesgue measure on $\g$, and $a \in \CT( T^\ast M \times \g)$.   Then $I(\mu)$ has the asymptotic expansion 
\bqn 
I(\mu) = (2\pi\mu)^\kappa L_0 + O(\mu^{\kappa+1}),  \qquad \mu \to 0^+.
\eqn
Here $\kappa$ is the dimension of an orbit of principal type in $M$, and the leading coefficient is given by 
\bq
\label{eq:L0}
L_0=\int_{\mathrm{Reg}\, \mathcal{C}} \frac {a(\eta,X)}{|\mathrm{Hess}  \, \psi(\eta,X)_{N_{(\eta,X)}\mathrm{Reg}\, \mathcal{C}}|^{1/2}} \d(\mathrm{Reg}\, \mathcal{C})(\eta,X),
\eq
where $\mathrm{Reg}\, \mathcal{C}$ denotes the regular part of the critical set $\mathcal{C}=\Crit(\psi)$ of $\psi$, and $\d(\mathrm{Reg}\, \mathcal{C})$ the induced Riemannian measure. 
In particular, the integral over $\mathrm{Reg}\, \mathcal{C}$ exists.
\end{theorem}
\begin{remark}
Note that equation \eqref{eq:L0} in particular means that the obtained asymptotic expansion for $I(\mu)$ is independent of the explicit partial resolution we used.
\end{remark}
\begin{proof}
By the considerations at the beginning of this section, one has 
\bqn 
I(\mu) = (2\pi\mu)^\kappa L_0 + O(\mu^{\kappa+1}),  \qquad \mu \to 0^+,
\eqn
where $L_0$ is given as a sum of integrals of the form \eqref{eq:L}. It therefore remains to show the equality \eqref{eq:L0}. For this, let us remark that since $M$ is compact, $T^\ast M$ is a paracompact manifold, admitting a Riemannian metric, so that $T^\ast M$ is  a metric space with metric $|\cdot|$.  We  introduce now certain cut-off functions for the singular part $\Sing \Omega$ of $\Omega$. Let $K$ be a compact subset in $T^\ast M$, $\epsilon >0$, and consider the set
 \bqn
 (\Sing \Omega \cap K)_{\epsilon}=\mklm{\eta \in T^\ast M : |\eta-\eta'| < \epsilon \text{ for some } \eta' \in \Sing \Omega}.
 \eqn
By using a partition of unity, one can  show the existence of a  test function $u_\epsilon \in \CT((\Sing \Omega \cap K)_{3\epsilon})$ satisfying $u_\epsilon =1$ on $(\Sing \Omega \cap K)_\epsilon$,  see H\"{o}rmander \cite{hoermanderI},  Theorem 1.4.1. We then have the following
\begin{lemma}\label{lemlim}
Let $a \in \CT(T^\ast M \times \g)$, $K$ be a compact subset in $T^\ast M$ such that $\supp_\eta a \subset K$, and $u_\epsilon$ as above. Then the  limit 
\bq
\label{eq:MC}
\lim_{\eps \to 0}  \int_{\Reg \mathcal{C}}\frac{ [a (1-u_\eps)] (\eta,X)}{|\det  \, \psi'' (\eta,X)_{|N_{(\eta,X)}\Reg \mathcal{C}} |^{1/2}} d(\Reg \mathcal{C})(\eta,X)
\eq
exists and is equal to $L_0$, where $d(\Reg {\mathcal{C}})$ is the induced Riemannian measure on $\Reg {\mathcal{C}}$.
\end{lemma}
\begin{proof}
We define
\bqn
I_\eps(\mu)=\int_{T^\ast M} \int_{\g}  e^{\frac{i}{\mu} \psi(\eta,X) }  [a(1-u_\eps)](\eta,X) \, dX \, d\xi \, dx.
\eqn
Since $(\eta,X) \in \Sing {\mathcal{C}}$ implies $ \eta \in \Sing \Omega$, a direct application of the generalized stationary phase theorem for fixed $\eps >0$ gives
\bq
\label{eq:asympt}
| I_\eps(\mu)- (2\pi \mu)^\kappa L_0(\eps) | \leq C_\eps \mu^{\kappa+1},
\eq
where $C_\eps>0$ is a constant depending only on $\eps$, and
\bqn
L_0(\eps)=  \int_{\Reg {\mathcal{C}}}\frac{ [a(1-u_\eps)] (\eta,X)}{|\det  \, \psi'' (\eta,X)_{|N_{(\eta,X)}\Reg {\mathcal{C}}} |^{1/2}} d(\Reg {\mathcal{C}})(\eta,X).
\eqn
On the other hand, applying our previous considerations to  $I_\eps(\mu)$ instead of $I(\mu)$, we obtain again an asymptotic expansion of the form \eqref{eq:asympt} for $I_\eps(\mu)$, where now the first coefficient is given by a sum of integrals of the form \eqref{eq:L} with $a$ replaced by $a(1-u_\eps)$. Since the first term in the asymptotic expansion \eqref{eq:asympt} is uniquely determined, the two expressions for $L_0(\eps)$ must be identical. The statement of the lemma now follows by  the Lebesgue theorem on bounded convergence.
\end{proof}
Note that existence of the limit in  \eqref{eq:MC} has been established by  partially resolving the  singularities of the critical set $\mathcal{C}$, the corresponding limit being given by $L_0$. Let now $a ^+ \in \CT(T^\ast M\times \g), \R^+)$. Since one can assume that  $|u_\epsilon| \leq 1$, the lemma of Fatou implies that 
\bqn 
  \int_{\Reg \mathcal{C}} \lim_{\eps \to 0}  \frac{ [a^+ (1-u_\eps)] (\eta,X)}{|\det  \, \psi'' (\eta,X)_{|N_{(\eta,X)}\Reg \mathcal{C}} |^{1/2}} d(\Reg \mathcal{C})(\eta,X)
\eqn
is mayorized by the limit  \eqref{eq:MC}, with $a$ replaced by $a^+$. Lemma \ref{lemlim} then implies that
\bqn 
  \int_{\Reg \mathcal{C}} \frac{ a^+ (\eta,X)}{|\det  \, \psi'' (\eta,X)_{|N_{(\eta,X)}\Reg \mathcal{C}} |^{1/2}} d(\Reg \mathcal{C})(\eta,X) < \infty.
\eqn
Choosing now $a^+$ to be equal  $1$ on a neighborhood of the support of $a$, and applying the theorem of Lebesgue on bounded convergence to the limit \eqref{eq:MC}, we obtain equation \eqref{eq:L0}. 
\end{proof}

\providecommand{\bysame}{\leavevmode\hbox to3em{\hrulefill}\thinspace}
\providecommand{\MR}{\relax\ifhmode\unskip\space\fi MR }
\providecommand{\MRhref}[2]{%
  \href{http://www.ams.org/mathscinet-getitem?mr=#1}{#2}
}
\providecommand{\href}[2]{#2}

%\bibliography{bibliography}
%\bibliographystyle{amsplain}

\end{document}